\pdfoutput=1
\documentclass[english, 11pt, a4paper]{article} 
\usepackage[left= 3.5cm, right=3.5cm, bottom=2.75cm,top=2.5cm]{geometry}
\usepackage{setspace} 
\usepackage[T1]{fontenc}
\usepackage[utf8]{inputenc}
\usepackage[light]{CormorantGaramond}
\usepackage[euler-hat-accent,euler-digits]{eulervm}
\usepackage{bm}
\usepackage{eucal}
\usepackage{eucal}
\usepackage[eulergreek]{mathastext}
\usepackage{appendix}
\usepackage[nottoc, notlof, notlot]{tocbibind}
\usepackage{amsmath,amsfonts,amssymb}
\usepackage{amsthm}
\usepackage{letltxmacro}
\LetLtxMacro\amsproof\proof
\LetLtxMacro\amsendproof\endproof
\usepackage{thmtools}
\usepackage[overload]{empheq}
\allowdisplaybreaks
\usepackage[
    unicode=true,
    pdfstartview=FitV,
    colorlinks=true,
    citecolor=Leturquoise,
    linkcolor=Leturquoise,
    urlcolor=MFCB,
    linktoc=page,
    hyperindex=true,
    pdfcreator={},
]{hyperref}
\usepackage{cleveref}
\crefname{equation}{Equation}{Equations}
\crefname{figure}{Figure}{Figures}
\Crefname{figure}{Figure}{Figures}
\crefname{section}{Section}{Sections}
\crefname{remark}{Remark}{Remarks}
\crefname{Claim}{Claim}{Claims}
\crefname{Def}{Definition}{Definitions}
\crefname{Th}{Theorem}{Theorems}
\crefname{Cor}{Corollary}{Corollaries}
\crefname{Prop}{Proposition}{Propositions}
\crefname{Lmm}{Lemma}{Lemmas}
\crefname{Ex}{Example}{Examples}
\crefname{Exs}{Examples}{Examples}
\crefname{CEx}{Counter-example}{Counter-example}
\crefname{Q}{Question}{Questions}
\crefname{Rq}{Remark}{Remarks}
\usepackage{babel}
\usepackage{adforn}
\usepackage{fancyhdr}
\usepackage[titles]{tocloft}

\usepackage{enumitem}
\setlist{nosep}
\usepackage[x11names,table]{xcolor} 
\usepackage{booktabs}
\definecolor{MCB}{cmyk}{0,0.03,0.08,0.26} 
\definecolor{MFCB}{cmyk}{0,0.06,0.20,0.6} 
\colorlet{Leturquoise}{DeepSkyBlue4}
\colorlet{TurquoiseClair}{DeepSkyBlue4!20}
\colorlet{MyOrange}{DarkOrange3!85}
\usepackage[pdftex]{graphicx}
\usepackage{subcaption} 
\usepackage{tikz} 
\usetikzlibrary{calc,math,patterns,decorations.pathmorphing}
\usetikzlibrary{arrows.meta} 
\usepackage{tcolorbox}
\newtcolorbox{HypBox}{colframe=black, colback=white,
  boxrule=0.0pt,bottomrule=1pt,toprule=1pt,outer arc=0pt,arc=0pt,
  before skip=0.2cm, after skip=0.2cm,
  halign=left}
\newtcolorbox{IP}{colframe=black,colback=white, arc=0mm, bottomrule=0pt,
  toprule=0pt, rightrule=0pt, leftrule=5pt, halign=left}
\newtcolorbox{Todo}{colframe=DeepSkyBlue4,colback=DeepSkyBlue4!10,halign=left}
\newcommand{\deffont}[1]{\textbf{#1}}
\newcommand{\bref}[1]{\cref{#1} on \cpageref{#1}}
\AtBeginDocument{%
  \LetLtxMacro\proof\amsproof
}
\declaretheorem[
style=definition,
thmbox={style=S,underline=false,bodystyle=\normalfont \noindent,thickness=0.3pt},
name=Definition,
numberwithin=section,
refname={Definition,Definitions},
Refname={Definition,Definitions}]{Def}

\declaretheorem[
style=plain,
thmbox={style=S, underline=false, bodystyle=\normalfont \noindent},
name=Theorem,
sibling=Def,
refname={Theorem,Theorems},
Refname={Theorm,Theorems}]{Th}

\declaretheorem[
style=plain,
thmbox={style=S,underline=false, bodystyle=\normalfont \noindent},
name=Proposition,
sibling=Def,
refname={Proposition,Propositions},
Refname={Propositions,Propositions}]{Prop}

\declaretheorem[
style=plain,
thmbox={style=S,underline=false,bodystyle=\normalfont \noindent},
name=Lemma,
sibling=Def,
heading={Lemma},
refname={Lemma,Lemmas},
Refname={Lemma,Lemmas}]{Lmm}
\theoremstyle{definition}
\newtheorem{Ex}[Def]{Example}

\newtheorem{Q}[Def]{Question}
\newtheorem{Claim}[Def]{Claim}
\newtheorem{Rq}[Def]{Remark}

\numberwithin{equation}{section}

\newcommand{\mbfe}{\boldsymbol{e}}
\newcommand{\mbff}{\boldsymbol{f}}
\newcommand{\mbfg}{\boldsymbol{g}}

\newcommand{\mbfL}{\boldsymbol{L}}
\newcommand{\mbfp}{\boldsymbol{p}}
\newcommand{\bN}{\mathbb{N}}
\newcommand{\bZ}{\mathbb{Z}}

\newcommand{\bR}{\mathbb{R}}

\newcommand{\calC}{\mathcal{C}}

\newcommand{\calG}{\mathcal{G}}
\newcommand{\calH}{\mathcal{H}}
\newcommand{\cH}{\mathcal{H}}

\newcommand{\calS}{\mathcal{S}}

\newcommand{\neutre}{\mathbf{e}}
\newcommand{\diam}[1]{\mathrm{diam}\left( {#1} \right)}

\newcommand{\FLL}{F^\prime}
\newcommand{\SLL}{\Sigma^\prime}
\newcommand{\Gs}{\Gamma_m}
\newcommand{\Gsp}{{\Gamma^\prime}_m}
\newcommand{\Ds}{\Delta_m}
\newcommand{\WDs}{\Gs\wr\bZ} 
\newcommand{\BZ}{\Delta}

\newcommand{\hS}{\Sigma}
\newcommand{\range}{\mathrm{range}}

\newcommand{\supp}{\mathrm{supp}}
\newcommand{\frakp}{\mathfrak{p}}

\newcommand{\profile}{I}
\newcommand{\landing}{\mathfrak{l}}
\newcommand{\Landing}{\mathfrak{L}}
\newcommand{\rhoaff}{\tilde{\rho}}

\newcommand{\Clm}{c_1}
\newcommand{\CTn}{c_{\BZ}}
\newcommand{\CTnn}{C_{\BZ}}

\newcommand{\CSnn}{C^\prime_{\BZ}}

\usepackage{titling}
\title{\texorpdfstring{\textsc{Sofic approximations and quantitative
      measure couplings}}{Sofic approximations and quantitative
      measure couplings}}
\author{Amandine Escalier}
\date{\today}
\hypersetup{ 
pdfauthor={Amandine Escalier},
pdftitle={Sofic approximations and quantitative measured couplings},
pdfsubject={Orbit equivalence},
pdfkeywords={Følner, lamplighter, diagonal product, measure, equivalence, groups}
}
\begin{document}
\maketitle
\begin{abstract}
  We build quantitative measure subgroup couplings from a Brieussel-Zheng
  diagonal product to a lamplighter group.
  We use them to answer the inverse problem of the quantification, namely find
  a group admitting a measure subgroup coupling with a prescribed group
  with prescribed quantification, in the case of the lamplighter group.


  
\end{abstract}
\setcounter{tocdepth}{2}
\tableofcontents
\newpage

\section{Introduction} \label{sec:Intro}
A recurring theme in group theory is the description of large-scale behaviour of
groups and their geometry. A well known example is the study of groups up to
quasi-isometry: it describes the large-scale or “coarse” geometry from the
\emph{metric} point of view. A \emph{measure} analogue of quasi-isometry –
called measure equivalence – was introduced by Gromov in \cite{Gro}. Two groups
$G$ and $H$ are \deffont{measure equivalent} if they both act freely and measure
preservingly on a standard measure space $(\Omega,m)$, such that the actions
commute and each action
admits a fundamental domain of finite measure.
A first elementary illustration of measure equivalent groups is given by
lattices in a common locally compact group. We refer to
\cite{Gab-survey,Fur-survey} for surveys on the topic.

In some cases, measure equivalence can show remarkable rigidity properties.
For instance Furman proved \cite{FurmanME} that any countable group which is
measure equivalent to a lattice in a simple Lie group $G$ of higher rank is
commensurable (up to finite kernel) to a lattice in $G$.
More recently
Guirardel and Horbez \cite{GuirardelHorbez} showed that for $n\! \geq\!
3$, any countable group that is measure equivalent to $\mathrm{Out}(F_n)$ is
virtually isomorphic to $\mathrm{Out}(F_n)$. 
Completely opposite to the aforementioned results, a famous
theorem of Ornstein and Weiss \cite{OW80} implies that 
all infinite countable amenable groups are measure equivalent. In particular –~unlike
quasi-isometry~– measure equivalence does \emph{not} preserve coarse geometric
invariants.

To overcome this issue it is therefore natural to look for some refinements of
this equivalence notion. For example Kerr and Li \cite{KerrLi} offer to sharpen
it by considering the Shannon entropy of partitions associated to the actions of
the two groups. We focus here on the \emph{quantiative} version as
introduced by Delabie, Koivisto, Le Maître and Tessera \cite{DKLMT}.

\subsection{Quantitative measure couplings}
Let $G$ and $H$ be two discrete groups that are measure equivalent
over a standard measure space $(\Omega,m)$ and denote by $X_G$ and $X_H$
the fundamental domains associated to their respective actions on $\Omega$.
The latter actions are both denoted by “$*$”.
In this case we have natural actions of $G$ on $X_H$, and $H$ on $X_G$, both
denoted by “$\cdot$” where for a.e. $x\in X_{H}$ and all $g\in G$,
we define $g\cdot x$ to be the unique
element of $H* g*x$ contained in $X_{H}$. We refer to
\bref{fig:Defactioninduite} for an illustration. The action of $H$ on $X_G$ is defined
analogously. The corresponding cocycles
$\alpha : G \times X_H \rightarrow H$ and $\beta: H \times X_G \rightarrow G$
are defined by 
\begin{equation}\label{eq:DefCocycle}
  \alpha(g,x)=h \Leftrightarrow h* (g* x) \in X_H
  \quad \text{and}\quad
  \beta(h,x)=g \Leftrightarrow g* (h* x) \in X_G.  
\end{equation}
When $x\mapsto\alpha(g,x)$ and $x\mapsto\beta(h,x)$ are in $L^p$ for all $g\in
G$ and $h\in H$, we say that the groups are \deffont{$L^p$-measure equivalent}.

This refinement allowed for example
Bader, Furman and Sauer \cite{BFSIntegrability} to obtain a new rigidity result:
they showed that any group which is $L^1$-measure equivalent to a lattice in $SO(n, 1)$
for some $n\geq 2$ is virtually a lattice in $SO(n,1)$.
It also led Bowen to
prove, in the appendix of \cite{AustinBowen}, that volume growth is invariant
under $L^1$-measure equivalence.
Delabie, Koivisto, Le Maître and Tessera offered in \cite{DKLMT} to
extend this quantification to a family of functions larger than $\{x\mapsto x^p,
\ p\in [0,+\infty]\}$. They also defined a relaxed version of measure
equivalence called \emph{measure subgroup coupling} and showed the monotonicity
of the isoperimetric profile under quantitative measure subgroup couplings,
a statement we make precise below.

\paragraph{Quantitative measure subgroup couplings}
Let $G=\langle S_G \rangle$ and $H=\langle S_H \rangle$ be two finitely
generated groups. A \deffont{measure subgroup coupling} from $G$ to $H$ is a
triple $(\Omega,m,X_H)$ such that:
\begin{itemize}
\item $(\Omega,m)$ is a standard measure space equipped with commuting
  measure-preserving free actions of $G$ and $H$, such that each action
  admits a Borel fundamental domain;
\item $X_H$ is a Borel fundamental domain of finite measure for the
  action of $H$ on $\Omega$. 
\end{itemize}
Such couplings arise naturally from coarse embeddings, as explained in
\cref{Ex:CoarseEmbeddings}.
Remark that when $(\Omega,m,X_H)$ is a measure subroup coupling, one can define
the corresponding cocycle $\alpha$ as in \cref{eq:DefCocycle}.
We refer to \cref{Subsec:MeasureCouplings} for more details on these couplings.

In the following $\bR_+$ denotes the interval $[0,+\infty [$ and $\bR^*_+$
stands for $]0,+\infty[$. Now let $\varphi:\bR_+\rightarrow \bR_+$ be a non-decreasing map. We say that
the measure subgroup coupling $(\Omega,m,X_H)$ is
\deffont{$\varphi$-integrable} if the cocycle $\alpha$ is
{$\varphi$-integrable}, namely if for all $g\in G$ there exists $c_g>0$ such
that $x\mapsto \varphi\left(c_g|\alpha(g,x)|_{S_H}\right) $ is integrable on $X_H$.

\paragraph{Monotonicity of the isoperimetric profile}
\label{Nota:Preccurly}
If $f$ and $g$ are two non-decreasing real functions, write $f \preccurlyeq g$ if there
exists some constant $C>0$ such that $f(x)={O}\big(g(Cx)\big)$ as $x$
tends to infinity\footnote{That is to say, if there exists $M>0$ and $x_0\in \bR$ such that
  $f(x)\leq M g(Cx)$ for all $x\geq x_0$}. We write $f\simeq g$ if $f
\preccurlyeq g$ and $g\preccurlyeq f$.

Let $G$ be a group generated by a finite set $S_G$. If $A\subseteq G$ we
denote by $\partial_{S_G}A$ the $S_G$-boundary of $A$, namely the set of elements $g\in
A$ for which there exists $s\in S_G$ such that $gs \notin A$. We will also denote
by $|g|_{S_G}$ the word length of an element $g\in G$. Recall that the
\deffont{isoperimetric profile} of $G$ is defined as\footnote{We chose to adopt
  the convention of 
  \cite{DKLMT}. In \cite{BZ}, the isoperimetric profile is defined as
  $\Lambda_G=1/\profile_G$.} 
\begin{equation*}
  \profile_{G}(n):= \sup_{|A|\leq n} \frac{|A|}{|\partial_{S_G}A|}.
\end{equation*}
Remark that due to the Følner criterion, a group is amenable if and only if its
isoperimetric profile is unbounded. Hence we can see the isoperimetric profile
as a way to measure the amenability of a group: the faster $\profile_G$ tends to
infinity, the more amenable $G$ is. For example the isoperimetric profile of
$\bZ$ verifies $\profile_{\bZ}(n) \simeq n$ while the profile of a lamplighter
group $G:=\bZ/m\bZ \wr \bZ$ verifies $I_G(n)\simeq \log(n)$, see for
example \cite{Erschler}.

Finally, let $k\in \bN$. We say that a measure subgroup coupling
$(\Omega,m,X_H)$ is \deffont{at most $k$-to-one} if for every $x\in X_H$ the map
$g\mapsto g^{-1}*(g\cdot x)$ has pre-images of size at most $k$.
Delabie, Koivisto, Le Maître and Tessera showed {\cite[Theorem~4.4]{DKLMT}} that
if  
\begin{itemize}
\item $\varphi$ and $t\mapsto t/\varphi(t)$ are two non-decreasing maps from
  $\bR^*_+$ to itself;
\item and there exists an at most $k$-to-one $\varphi$-integrable measure subgroup
  coupling from $G$ to $H$, for some $k\in \bN$; 
\end{itemize}
then $\varphi\circ I_H \preccurlyeq I_G$\label{Th:ProfiletOE}. 
\medskip

This last result together with the work of Brieussel and Zheng \cite{BZ}
constructing amenable groups with
prescribed isoperimetric profile, led us to study the inverse problem of
the quantification.

\subsection{Background on the inverse problem}
The family of amenable groups being quite large, the evaluation of its diversity
is at the heart of numerous works. To do so, one can rely for example on
geometric quantities –~such as volume growth or the isoperimetric profile~–
or probabilistic data –~such as entropy or return probability of random walks.
In order to quantify this diversity, several results have been focusing on
inverse problems namely, given a prescribed behaviour, does there exist a group having
such behaviour?

For example, the question of the existence of a
group with intermediate growth has been solved by Grigorchuk
\cite{Grigorchuk}. Later on, Bartholdi and Erschler
exhibited uncountably many groups with intermediate growth
\cite{BarErschler,BarErschlerPermut}. Another illustration is the striking work
of Brieussel and Zheng \cite{BZ}, answering the inverse problem for a large
family of quantities, such as the isoperimetric profile, entropy, or equivariant
$L^p$-compression.

The inverse problem of the quantification formalised below can be seen
as a way to quantify the diversity of amenable groups from the \emph{ergodic} point of view.

\begin{Q}[Inverse Problem]\label{Q:InversePb}
  Given a group $H$ and a non-decreasing map $\varphi:\bR^*_+\rightarrow \bR^*_+$, does 
  there exist a group $G$ such that:
  \begin{itemize}
  \item there exists a $\varphi$-integrable measure subgroup coupling from $G$ to $H$;
  \item and this quantification is optimal, that is to say such that any other
    $\varphi^\prime$-integrable coupling verifies $\varphi^\prime
    \preccurlyeq \varphi$?
  \end{itemize}
\end{Q}

Note that the monotonicity of the isoperimetric profile recalled on 
\cpageref{Th:ProfiletOE} provides us with an upper bound to the possible optimal
integrability. For example when $H=\bZ$, we obtain that if there exists
a $\varphi$-integrable measure sugroup coupling from $G$ to $H$, then
$\varphi \preccurlyeq I_G$. More generally, for any group $H$ the question can
be rephrased as follows.

\begin{Q}\label{Q:IPProfil}
  Given a group $H$ and a non-decreasing function $\varphi$, does 
  there exist a group $G$ such that $I_G\simeq \varphi\circ I_H$, and such that
  there exists a $\varphi$-integrable measure subgroup coupling from
  $G$ to $H$?
\end{Q}

Relying on Brieussel-Zheng groups, we tackled these questions in
\cite{EscalierOEZ} for $H=\bZ$.
This article studies the case when $H$ is a lamplighter group.

\subsection{Main result}\label{Subsec:MainResults}
The main result of this article is \cref{Th:CouplageLL} below.

To prove it, we rely on the diagonal products of lamplighter groups defined by
Brieussel and Zheng \cite{BZ}, providing groups with prescribed isoperimetric
profile. Such groups can be defined for profiles of the form $I_G\simeq \rho \circ
\log$ where $\rho$ belongs to the following family of
functions:
\begin{equation}\label{Def:classC}
  \calC:=\left\{ \rho :[1,+\infty) \rightarrow [1,+\infty) \ \left\vert
    \begin{matrix}
      \rho \ \text{is continuous}, \\
      \rho \ \text{and} \ x\mapsto x/\rho(x)\ \text{are non-decreasing}
    \end{matrix} \right.
  \right\}.
\end{equation}
Note that contrarily to \cite{BZ}, we do not assume that $\rho(1)=1$.
We refer to {\bref{Rq:rho-un}} for an explanation.

\begin{Th}\label{Th:CouplageLL}~Let $m\geq 2$ be an integer and let $L_m:=(\bZ/m\bZ) \wr \bZ$.
  For all $\rho\in \calC$ there exists a group $G$ such that
  \begin{itemize}
  \item $I_G\simeq \rho \circ \log$; and
  \item there exists an at most $1$-to-one measure subgroup coupling from $G$ to
    $L_m$ which is $\rho^{1-\varepsilon}$-integrable for all $\varepsilon>0$.
  \end{itemize}
\end{Th}

Recall that the lamplighter group $L_m$ verifies $I_{L_m}\simeq \log$. By monotonicity
of the profile under quantitative measure subgroup coupling (see \cpageref{Th:ProfiletOE}),
if there exists an at most $k$-to-one 
$\varphi$-integrable measure subgroup coupling from $G$ to $L_m$, then
$\varphi \circ \log \preccurlyeq I_G=\rho\circ \log$.  Namely $\varphi$ has to
grow slower than $\rho$. Therefore, the integrability of the coupling given by
\cref{Th:CouplageLL} is almost optimal.

\paragraph{Structure of the paper and strategy of the proof} The next two
sections present the tools we 
use to build our couplings. In \cref{Sec:DefdeBZ} we introduce the necessary material on
Brieussel-Zheng diagonal products; background on measure couplings is recalled
in \cref{Sec:MeasureCouplings}. The coupling is then built in
\cref{Sec:Construction} and relies on the criterion given in \bref{Th:soficME}.

The strategy is the following: given a quantification $\rho\in \calC$ we chose
$G$ to be a diagonal product with isoperimetric profile $I_G\simeq \rho \circ
\log$. The goal is then to satisfy the conditions of \cref{Th:soficME}, namely to define
sofic approximations $(\calG_n)_n$ and $(\calH_n)_n$ in respectively $G$ and
$H$, then injections $\iota_n:\calG_n\rightarrow \calH_n$ such that $\iota_n$
respects the geometry –~see \bref{eq:CondsoficIntegrability} for the precise
statement.

The appropriate sofic approximations in $G$
is defined in \cref{Sec:DefOfTheSA}. We then define the injections
$\iota_n$ between them and show that these injections respect the geometry in
\cref{Sec:Qt}. 

\begin{center}
  \adforn{24}
\end{center}
\paragraph*{Acknowlegements}
The author thanks Romain Tessera and Jérémie Brieussel, under
whose supervision the work presented in this article was carried out.
She thanks them for suggesting the topic, sharing their precious insights and for
their abundant useful advice. She also thanks the referees for their helpful
comments.

The author is supported by the European Union (ERC, Artin-Out-ME-OA,
101040507). Views and opinions expressed are however those of the authors only
and do not necessarily reflect those of the European Union or the European
Research Council. Neither the European Union nor the granting authority can be
held responsible for them.

The author was also funded by the Deutsche
Forschungsgemeinschaft (DFG, German Research Foundation) – Project-ID
427320536 – SFB 1442, as well as under Germany’s 
Excellence Strategy EXC 2044 –390685587, Mathematics Münster:
Dynamics–Geometry–Structure.
\section{Background on diagonal products}\label{Sec:DefdeBZ}
In order for this article to be self contained, we
recall necessary material from \cite{BZ} concerning the definition of
\emph{Brieussel-Zheng diagonal products}: we give the definition of such
a group and recall some results concerning the range
of an element. Finally we present in \cref{Subsec:FromIPtoBZ}
the tools needed to recover such a diagonal product starting with a prescribed
isoperimetric profile. 

\subsection{Definition of diagonal products} \label{subsec:DefDelta}
Recall that the wreath product of a group $G$ with $\bZ$, denoted by $G \wr
\bZ$, is defined as
\begin{equation*}
  G\wr \bZ:= \oplus_{m \in \bZ} G \rtimes \bZ,
\end{equation*}
where $\bZ$ acts on $\oplus_{m\in \bZ}G$ by shifting the index.
An element of $G\wr \bZ$ is a pair $(f,t)$ where $f$ is a map from $\bZ$ to $G$ with finite support
and $t$ belongs to $\bZ$. We refer to $f$ as the \deffont{lamp configuration}
and $t$ as the \deffont{cursor.} Finally we denote by $\supp(f)$ the
\deffont{support} of $f$ which is defined as
$\supp(f):=\left\{x\in \bZ \ | \ f(x)\neq e_G \right\}$.

\paragraph{General definition}
Let $A$ and $B$ be two finite groups. Let $(\Gamma_m)_{m\in \bN}$ be a sequence
of groups such that each $\Gamma_m$ admits a generating set of the form $A_m \cup B_m$
where $A_m$ and $B_m$ are finite subgroups of $\Gamma_m$ isomorphic respectively
to $A$ and $B$. For $a \in A$ we denote $a_m$ the copy of $a$ in $A_m$ and
similarly for $B_m$. Finally let $(k_m)_{m \in \bN}$ be a sequence of integers such that $k_{m+1} \geq 2 k_m$
for all $m$. We define $\Ds=\WDs$ and endow it with the generating set  \label{Def:SDeltam}
\begin{equation*}
  \mathcal{S}_{\BZ_m}:=
  \left\{ \left(\neutre_{\Gamma_m},1 \right)\right\} \cup \Big\{ \big(a_m \delta_0,0 \big) \ | \ a_m \in A_m\Big\}
  \cup \Big\{ \big(b_m \delta_{k_m},0 \big) \ | \ b_m \in B_m\Big\}.
\end{equation*}

\begin{Def}[\cite{BZ}]\label{Def:BZGr}\index{Diagonal product} \index{Brieussel-Zheng’s
    diagonal product}
  The \deffont{diagonal product} associated to the sequences
  $(\Gamma_m)_{m\in \bN}$ and $(k_m)_{m\in \bN}$, is the subgroup $\BZ$ of $ \left(\prod_m 
    \Gamma_m\right)\wr \bZ$ generated by 
  \begin{equation*}
    \calS_{\BZ}:=
    \Big\{ \Big({ {\big(\neutre_{\Gamma_m}\big)}_m,1 }\Big)\Big\} \cup
    \Big\{ \big( {\left(a_m \delta_0  \right)}_m,0 \big) \ | \ a \in A\Big\} 
    \cup \Big\{ \big( {\left( b_m \delta_{k_m} \right)}_m,0 \big) \ | \ b \in B\Big\}.
  \end{equation*}
\end{Def}
An element of $\BZ$ is therefore of the form $\left( (g_m)_m,t \right)$, where $t\in \bZ$
and $g_m:\bZ \rightarrow \Gamma_m$ for all $m \in \bN$. In the following, we
will denote by $\mbfg$ the sequence $(g_m)_{m\in \bN}$.
For an example of an element in $\BZ$ we refer to \cite[Example 2.2]{EscalierOEZ}.

The subgroup $\BZ$ is uniquely determined by the sequence 
$(k_m)_{m\in \bN}$. Moreover, for any sequence of integers $(k_m)_m$ and any sequence of
groups $(\Gamma_m)_m$ as above, there exists a diagonal product $\BZ$.

\paragraph{Relative commutators subgroups}\label{subsec:Derivedfunctions}
From now on and following \cite[Assumption 2.1]{BZ}, we assume that $\langle{ 
  \langle{\left[A_m,B_m\right]}}  \rangle \rangle \backslash \Gs$ and $A \times
B$ are isomorphic for all $m\in \bN $, where $\langle{\langle{\left[A_m,B_m\right]}\rangle}\rangle$
is the normal closure of $\left[A_m,B_m\right]$. Similarily as in \cite[Notation
2.6]{BZ}, we denote by 
$\theta_m : \Gs \rightarrow \langle{ \langle{\left[A_m,B_m\right]}}  \rangle
\rangle \backslash \Gs $ the natural projection. Let $\theta^A_m$ and $\theta^B_m$
denote the composition of $\theta_m$ with the projection to $A_m$ and $B_m$
respectively. Now let $m\in \bN$ and define $\Gsp:= \langle{ \langle{
    \left[A_m,B_m\right]}} \rangle \rangle$. If $(g_m,t)$ belongs to $\Delta_m$,
then there exists a unique $g^{\prime}_m \ : \ \bZ \rightarrow \Gsp$ such that
$g_m(x)=g^{\prime}_m(x)\theta_m\big(g_m(x)\big)$ for all $x\in \bZ$.
We refer to \cite[Examples 2.3 and 2.4]{EscalierOEZ} for examples of computations of the
maps~$\theta_m(g_m)$ and of~$g^\prime_m$.

\paragraph{The expanders case}\label{subsec:Expanders}

In \cite[Proposition 4.4]{BZ}, Brieussel and Zheng estimate the isoperimetric profile
of any diagonal product. We will need the finest estimate given by \cite[Theorem
4.6]{BZ} obtained when
$(\Gamma_m)_{m\in \bN}$ is a family of \emph{expanders}. Recall that $(\Gamma_m)_{m\in \bN}$
is said to be a sequence of \deffont{expanders}\index{Expander} if the sequence
of diameters $(\diam{\Gamma_m})_{m\in \bN}$ is unbounded and if there exists
$c_0>0$ such that for all $m\in \bN$ and all $n\leq |\Gamma_m|/2$ the
isoperimetric profile verifies $I_{\Gamma_m}(n)\leq c_0$.

Now, consider a family $(\Gamma_m)_{m\in \bN}$ of expanders. Assume that there
exists $c>0$ such that for all $\ell \geq 1$ there exists $\Gamma_{\frakp(\ell)}$ satisfying
that $\frac{1}{c} \ell \leq \diam{\Gamma_{\frakp(\ell)}}\leq c \ell$. We can thus define a
“parametrization” by fixing a map $\ell \mapsto \Gamma_{\frakp(\ell)}$. Consider now
two non-decreasing sequences 
$(k_m)_{m\in \bN}$ and $(l_m)_{m\in \bN}$ of real numbers greater than $1$ and
denote by $\BZ$ the diagonal product 
associated to $(\Gamma_{\frakp(l_m)})_{m\in \bN}$ and $(k_m)_{m\in \bN}$. Then
$\BZ$ is uniquely determined by the data of $(l_m)_{m\in \bN}$ and $(k_m)_{m\in
  \bN}$. In what follows, we will abuse notation and write $\Gamma_m$ instead of
$\Gamma_{\frakp(l_m)}$. Moreover we will always make the following assumptions
when talking about diagonal products.
We refer to \cite[Example 2.3]{BZ} for an explicit example of
diagonal product satisfying \textbf{(H)}.

\begin{HypBox}\label{Hyp:H}
    \textbf{Hypotheses (H)}\\
  \begin{enumerate}
  \item\label{H:1} $q:=|A\times B|=12$;
  \item\label{H:5} $k_0=0$ and $l_0=1$ and $\Gamma_0=A_0\times B_0$;
  \item\label{H:2} $\kappa \geq 3$ and $(k_m)_{m\in \bN^*}$ is a sub-sequence of
    $(\kappa^m)_{m\in \bN}$;
  \item\label{H:3} $\lambda \geq 2$ and $(l_m)_{m\in \bN}$ is a sub-sequence of
    $(\lambda^m)_{m\in \bN}$;
  \item\label{H:4}  $(\Gamma_m)_{m\in \bN}$ is a sequence of expanders such that $\Gamma_m$
    is a quotient of $A*B$ and there exists $\Clm >0$ such that
    $\diam{\Gamma_m}\leq \Clm l_m$ for all $m\in \bN$;
  \item\label{H:6}  the natural quotient map $A_m\times B_m\rightarrow
    \langle{ \langle{\left[A_m,B_m\right]}}  \rangle \rangle \backslash \Gamma_m$
    is an isomorphism, where $\langle{\langle{\left[A_m,B_m\right]}\rangle}\rangle=:\Gamma^\prime_m$
    is the normal closure of $\left[A_m,B_m\right]$.
  \end{enumerate}
\end{HypBox}
Recall from \cite[page 9]{BZ}, that in this case there exist $c_1$, $c_2>0$ such
that, for all $m$
\begin{equation}
  \label{eq:encadrementlngammap}
  c_1 l_m -c_2 \leq \ln \left\vert {\Gamma_m}\right\vert \leq c_1l_m + c_2.
\end{equation}

Finally, we adopt the convention of \cite[Notation 2.2]{BZ} and allow $k_m$
to take the value~$+\infty$. In this case $\BZ_m$ is the trivial group. In
particular when $k_1=+\infty$ the diagonal product $\BZ$ corresponds to the
usual lamplighter group $(A\times B)\wr \bZ$. 
\subsection{Range of an element}\label{Subsec:RangeDelta}
In this section we recall the notion of \emph{range} of an element $(\mbfg,t)$ in $\BZ$.
We denote by $\pi_{2} : \Delta \rightarrow \bZ$ the projection on the
second factor, ie. the map that sends $(\mbfg,t)\in \BZ$ to the value of its
cursor~$t$.

\begin{Def}[{\cite[Definition 2.8]{BZ}}]\label{Def:Range}
  If $w=s_1\ldots s_m$ is a word over ${\mathcal{S}}_\BZ$ we define its
  \deffont{range}\index{Range!Of a word} as
  \begin{equation*}
    \range(w) := \left\{\pi_2\left(\prod^{i}_{j=1}s_j \right) \,
      : \, i=0,\ldots,m \right\}.
  \end{equation*}
\end{Def}

The range is a finite subinterval of $\bZ$. It represents the set of sites
visited by the cursor. 
\begin{Def}\label{Def:Rangedelta}
  The \deffont{range} of an element $(\mbfg,t)\in \Delta$ is defined as the
  minimal length interval obtained as the range of a word
  over $\mathcal{S}_{\BZ}$ representing $(\mbfg,t)$.
\end{Def}
In what follows we will consider elements that can be written as a word with range
in an interval of the form $[0,n]$, where $n$ belongs to $\bN$. Therefore,
when there is no ambiguity we will denote by $\range(\mbfg,t)$ this interval,
namely $\range(\mbfg,t)=[0,n]$. For all $n\in \bN$ we
denote by $\landing(n)$ the integer such that $k_{\landing(n)}\leq n < 
k_{\landing(n)+1}$. In particular, for all $m>\landing(n)$ we have $g^\prime_m=\neutre_{\Gamma^\prime_m}$.

\begin{Claim}[{\cite[Fact 2.9]{BZ}}]\label{Claim:Data}
  An element $(\mbfg,t) \in \BZ$ is uniquely determined by $t$, $g_0$ and the sequence
  $ (g^{\prime}_m)_{m\leq \landing(\range(\mbfg,t))}$. 
\end{Claim}
We refer to \cite[Examples 2.9 and 2.11]{EscalierOEZ} for illustrated examples.

\subsection{From the isoperimetric profile to the
  group}\label{Subsec:FromIPtoBZ}
We saw how to define a diagonal product from two sequences $(k_m)_m$ and
$(l_m)_m$. Now, given a prescribed map $\varphi$, the goal is to
define as in \cite[Appendice B]{BZ}, a diagonal product $\Delta$ such that $I_{\Delta}\simeq \varphi$. 

\paragraph{\texorpdfstring{Definition of $\Delta$}{Definition of the
    diagonal product}}
Recall that in the particular case of expanders, a
diagonal product $\BZ$ is uniquely determined by the sequences 
$(k_m)_{m\in \bN}$ and $(l_m)_{m\in \bN}$, where $l_m$ corresponds to the
diameter of $\Gs$.
Thus, starting from a prescribed function $I$, the goal is to define sequences
$(k_m)_{m\in \bN}$ and $(l_m)_{m\in \bN}$ such that the corresponding $\BZ$ 
verifies $\profile_{\BZ} \simeq I$. Brieussel and Zheng’s construction of
diagonal products is possible when the profile $I$ is of the form $I\simeq \rho
\circ \log$ where $\rho$ belongs to the following set $\calC$,
\begin{equation*}
  \calC:=\left\{ \rho :[1,+\infty) \rightarrow [1,+\infty) \ \left\vert
    \begin{matrix}
      \rho \ \text{continuous}, \\
      \rho \ \text{and} \ x\mapsto x/\rho(x)\ \text{non-decreasing}
    \end{matrix} \right.
  \right\}.
\end{equation*}
Equivalently this is the set of functions $\rho$ satisfying 
\begin{equation}
  \label{Rq:IneqTildeRho}
 \left( \forall x,c\geq 1 \right) \quad \rho(x)\leq \rho(cx) \leq c\rho(x).  
\end{equation}
So let $\rho \in \calC$. Note that unlike \cite{BZ}, we do not assume here that
$\rho(1)=1$. We refer to \cref{Rq:rho-un} for more details.

Combining \cite[Proposition B.2 and
Theorem 4.6]{BZ} leads to the following proposition. Remember that with our
convention the isoperimetric profile considered in \cite{BZ} corresponds to
$1/I_{\BZ}$.

\begin{Prop}[\cite{BZ}]\label{Prop:DefdeDelta}
  Let $\kappa, \lambda \geq 2$. For any $\rho \in \calC$ there exists a
  subsequence $(k_m)_{m\in \bN^*}$ of $(\kappa^n)_{n\in \bN}$ and a subsequence
  $(l_m)_{m\in \bN}$ of $(\lambda^n)_{n\in \bN}$ such that the group $\BZ$ defined in
  \cref{subsec:Expanders} verifies $\profile_{\BZ} \simeq \rho \circ \log$.
\end{Prop}

\begin{Ex}[{\cite[Example 4.5]{BZ}}]\label{Ex:Alpha} Let $\alpha >0$. If
  $\rho(x):=x^{1/(1+\alpha)}$ then the diagonal 
  product $\BZ$ defined by $k_m=\kappa^m$ and $l_m=\kappa^{\alpha m}$ verifies
  $I_{\BZ}\simeq \rho \circ \log$. 
\end{Ex}
\begin{Ex}\label{Ex:logloglog} If $\rho=\log$ then 
  the diagonal product $\BZ$ defined by $k_m=\kappa^m$ and
  $l_m=\kappa^{\kappa^m}$ verifies $I_{\BZ}\simeq \log\circ \log$.
\end{Ex}

Recall from below \cref{eq:encadrementlngammap} that we allow $k_m$ to take the value $+\infty$.
\begin{Ex}
  If $\rho(x)=x$ then the diagonal product defined by $l_m=1$ for all $m$ and
  $k_m=+\infty$ for all $m\geq 1$ verifies $\BZ=(A\times B)\wr \bZ$ and
  $I_\BZ\simeq \log$.  
\end{Ex}
\paragraph{Technical tools}\label{subsec:TechTools}
Now let us recall the intermediary functions
defined in \cite[Appendix B]{BZ} and some of their properties.

Let $\rho \in \calC$. %
The construction of
a group with the given isoperimetric profile $\rho \circ \log$ is
based on the approximation of $\rho$ by a piecewise linear function $\rhoaff$.
For the quantification of our measure coupling, many of our computations will use
$\rhoaff$ and some of its properties.
The following lemma recalls {\cite[Equation (52) and Lemma B1]{BZ}} defining $\rhoaff$.

\begin{Lmm}[\cite{BZ}]
  \label{Prop:B2} Let $\rho\in \calC$. 
  Let $(k_m)_m$ and $(l_m)_m$ given by \cref{Prop:DefdeDelta} and $\BZ$ be the
  corresponding diagonal product. Let $\rhoaff$ be the map defined by 
  \begin{equation}\label{Def:Tilde}
    \rhoaff(x):= \begin{cases}
      x/l_m & \text{if} \ x\in [k_ml_m,k_{m+1}l_m],\\
      k_{m+1} & \text{if} \ x\in [k_{m+1}l_m,k_{m+1}l_{m+1}],
    \end{cases}
  \end{equation}
  Then, there exists $c>0$ such that $\rho(x)/c\leq \rhoaff(x)\leq c\rho(x)$ for
  all $x\geq 1$.
\end{Lmm}

\begin{Rq}\label{Rq:rho-un}
  In \cite{BZ}, Brieussel and Zheng further request that $\rho(1)=1$ to ensure that the
  constant $c$ does not depend on $(k_m)_m$ nor $(l_m)_m$. Since we do not
  need the universality of the contant $c$ in the present article, we drop the
  assumption. Applying the results of \cite{BZ} to
  $\frac{1}{\rho(1)}\rho$ instead of $\rho$ leads to the above \cref{Prop:B2}. 
\end{Rq}

\section{Preliminaries on measure couplings}\label{Sec:MeasureCouplings}
This section recalls some material from \cite[Section~2]{DKLMT} regarding
measure couplings. We start with
general definitions on measure equivalence and measure subgroup couplings, then
turn to their quantitative version. Finally we present the two couplings
building techniques we use in order to show our main theorem.

\subsection{Measure couplings}\label{Subsec:MeasureCouplings}
A \deffont{standard Borel space} $(\Omega, \mathcal{B}(\Omega))$ is a measurable
space whose $\sigma$-algebra $\mathcal{B}(\Omega)$ consists of the the Borel
subsets coming from some Polish (separable and completely metrisable) topology
on $\Omega$. 
A \deffont{standard measure space} $(\Omega,m)$ is a standard Borel space
$(\Omega,\mathcal{B}(\Omega))$ equipped with a non-zero 
measure~$m$.

A \deffont{measure-preserving action}\index{Measure!Preserving action} 
of a discrete countable group $G$ on a measure space $(\Omega,m)$ is an action of
$G$ on $\Omega$, denoted here by “$*$”, such that the map sending $(g,x)$ to $g*x$ is a Borel map and
$m(E)=m(g* E)$ for all $E\subseteq \mathcal{B}(\Omega)$ and all $g\in G$.
A measure-preserving action of $G$ on $(\Omega,m)$ is
\deffont{free},\index{Free action} if for almost every $x\in\Omega$ we have $g* x =
x$ if and only if $g=\neutre_G$. A \deffont{fundamental domain} for an action of
$G$ on $(\Omega,m)$ is a Borel subset $X_G\subseteq \Omega$ which intersects
almost every $G$-orbit at exactly one point: in other words, there is a full
measure $G$-invariant Borel set $\Omega^*\subseteq \Omega$ such that for all
$x\in \Omega^*$, the intersection $\mathrm{Orb}_G(x)\cap X_G$ is a singleton. A
measure preserving action of a countable group on a standard measure space is
\deffont{smooth} if it admits a fundamental domain.

Let $G$ and $H$ be two countable groups. Following the terminology of
\cite{DKLMT}, we call a \deffont{measure equivalence coupling} from $G$ to $H$ a
quadruple $(\Omega,X_G,X_H,m)$ such that $G$ and $H$ both act freely, measure
preservingly, smoothly and commutingly on the standard measure space
$(\Omega,m)$, and $X_G$ and $X_H$ are fundamental domains of finite measure for
the $G$- and the $H$-action on $\Omega$, respectively. Note that in this case
$(\Omega,m)$ is a $\sigma$-finite measure space.

An elementary example of measure equivalence coupling is given by considering a
countable group $G$ endowed with the counting measure $m$, taking $G=H$ and
as actions the left and right translations of $G$ on itself. The corresponding
coupling is then $(G,\{e_G\},\{e_G\},m)$.

When $G$ is an infinite index subgroup of $H$, however, the left action by
translation of $G$ on $H$ does not admit any fundamental domain of finite
measure. This is why Delabie, Koivisto, Le Maître and Tessera introduced the
notion of \emph{measure subgroup coupling}, relaxing the condition on the
fundamental domain of the $G$-action.

\begin{Def}[{\cite[Definition 2.4]{DKLMT}}]\label{Def:MSGr}
  Let $G$ and $H$ be two countable groups. A \deffont{measure subgroup
    coupling} from $G$ to $H$ is a triple $(\Omega,X_H,m)$ such that:
  \begin{itemize}
  \item $(\Omega,m)$ is a standard measure space equipped with smooth commuting\\
    measure-preserving free actions of $G$ and $H$; 
  \item and the $G$-action on $\Omega$ admits a Borel fundamental domain;
  \item and $X_H$ is a Borel fundamental domain of finite measure for the
    action of $H$ on $\Omega$. 
  \end{itemize}
\end{Def}
Remark that a measure equivalence coupling $(\Omega,X_G,X_H,m)$ from $G$ to $H$ induces two measure
subgroup couplings, namely $(\Omega,X_H,m)$ from $G$ to $H$ and $(\Omega,X_G,m)$ from $H$ to $G$.
Also, as suggested above, when $G$ is a subgroup of $H$, a measure subgroup coupling
grom $G$ to $H$ is given by $(H,\{e_H\},m)$, where $m$ denotes again the counting
measure, and $G$ acts on $H$ by left translation, and $H$ on itself by right
translation. 

We saw that two lattices in a common locally compact group are measure
equivalent. Analogously, another example of measure subgroup coupling is given by
considering $\calG$ a locally compact group endowed with a Haar measure $m$,
then taking $G\leq \calG$ to be a discrete group and $H\leq \calG$ to be a
lattice. Letting $G$ and $H$ act respectively by left and right translations on
$\calG$, and denoting by $X_H$ a fundamental domain for the $H$-action then produces
the triple $(\calG,X_H,m)$ which is a measure subgroup coupling from $G$ to $H$.
\smallskip

Finally, let $k\in \bN$. We say that a measure subgroup coupling
$(\Omega,X_H,m)$ from $G$ to $H$ is at \deffont{most $k$-to-one} if for every
$x\in X_H$ the map $g\mapsto g^{-1}*(g\cdot x)$ has pre-images of size at most
$k$. Such couplings arise naturally from coarse embeddings, as we describe
below.

\begin{Ex}[Coarse embeddings and regular maps]\label{Ex:CoarseEmbeddings}
  A map $f\colon G \rightarrow H$ is a \deffont{regular map} if it
  is Lipschitz and verifies $\sup_{h\in H}|f^{-1}(\{h\})|<+\infty$. Coarse
  embeddings, for example, are regular maps. Recall that a \deffont{coarse
    embedding} is a map $f\colon G\rightarrow H$ such 
  that there exists two proper non-decreasing functions $\rho_-,\rho_+:\bR_+
  \rightarrow \bR_+$ such that $\lim_{t\rightarrow \infty}\rho_{-}(t)=\infty$,
  and satisfying for all $g_1,g_2\in G$,
  \begin{equation*}
    \rho_{-} \left(d_G\left( g_1,g_2 \right)  \right)\leq
    d_H\left( f\left( g_1 \right), f\left( g_2 \right) \right)
    \leq \rho_{+} \left(d_G\left( g_1,g_2 \right)  \right).
  \end{equation*}

  Delabie et al. showed \cite[Theorem~5.4]{DKLMT} that if $G$ is amenable and if
  there exists a regular map from $G$ to $H$, then there exists $k\in \bN$
  such that there is an at most $k$-to-one measure subgroup coupling from $G$ to
  $H$. A more precise version of this statement is given in \bref{Th:RegularEmbeddings},
  where we add the integrability of the measure subgroup.
\end{Ex}

Let us now define the quantitative version of such couplings.

\subsection{Quantitative couplings}\label{Sec:DefQt}

\subsubsection{Definitions and examples}
Recall that if a finitely generated group $G$ acts on a space $\Omega$ and if $S_G$
is a finite generating set of $G$, one can define the \deffont{Schreier graph} associated to this
action. It is the graph whose set of vertices is $\Omega$ and set of edges is
$\{(x,s* x) \, | \, s\in S_G, \,x\in \Omega \}$. This graph is endowed with a natural metric
$d_{S_G}$ on the orbits, fixing the length of an edge to one. Remark that if $S^\prime_G$ is another
generating set of $G$ then there exists $C>0$ such that for all $x\in \Omega$
and all $g\in G$
\begin{equation*}
 \frac{1}{C}d_{S_G}(x,g * x) \leq d_{S^{\prime}_G}(x,g* x) \leq Cd_{S_G}(x,g* x).
\end{equation*}

Finally if $(\Omega,X_H,m)$ is a measure subgroup coupling from $G$ to $H$,
we have a natural action of $G$ on $X_H$
denoted by “$\cdot$” and illustrated on \cref{fig:Defactioninduite}, where for a.e. $x\in
X_{H}$ and all $g\in G$ we define $g\cdot x$ to be the unique element of $H* g*
x$ contained in $X_{H}$, \emph{viz.}
\begin{equation*}
 \{ g \cdot x\}= \left(H* g* x  \right)  \cap X_H.
\end{equation*}

\begin{Def}[{\cite[Def. 2.20]{DKLMT}}]\label{Def:QuantMsgr}
  Let $\varphi:\bR_+\rightarrow \bR_+$ be a non-decreasing map. Let $G$ and $H$
  be two countable groups and denote by $S_H$ a finite generating set of $H$.
  A measure subgroup coupling
  $(\Omega,X_H,m)$ from $G$ to $H$ is said to be \deffont{$\varphi$-integrable} if
  for all $g\in G$ there exists $c_g>0$ such that
  \begin{equation*}
    \int_{X_H} \varphi\left({c_g} d_{S_H}(g* x,g \cdot x)
    \right)\mathrm{d}m(x) <+\infty.
  \end{equation*}
\end{Def}
\begin{figure}[htbp]
  \centering
  \includegraphics[width=\textwidth]{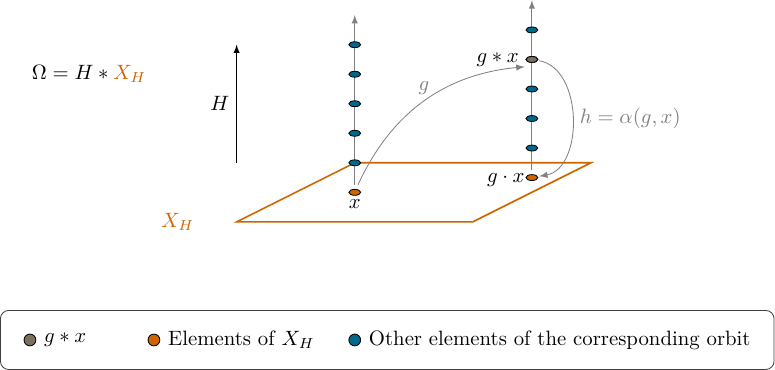} 
  \caption{Definition of $g\cdot x$, adapted from \cite{EH}}
  \label{fig:Defactioninduite}
\end{figure}

The constant $c_g$ in the definition is introduced for the 
integrability to be independent of the choice of generating set $S_H$.
If $\varphi(x)=x^p$ we will sometimes 
talk of \deffont{$\mbfL^{\mbfp}$-integrability} instead of
$\varphi$-integrability. In particular, $L^0$ means that no integrability
assumption is made. The coupling is said to be \deffont{$\mbfL^\infty$-integrable}
if $x\mapsto d_{S_H}(g* x,g \cdot x)$ is essentially bounded for all $g\in G$. 
\begin{Rq}[Reformulation using cocycles]
  When $(\Omega,X_H,m)$ is a measure subgroup coupling from $G$ to $H$, the
  associated \deffont{cocycle} $\alpha:G\times X_H\rightarrow H$ is defined as
  being the map that verifies $\alpha(g,x)* g* x=g\cdot x$, for all $g\in G$ and
  a.e. $x\in \Omega$. In other words, $\alpha(g,x)$ is the unique element of
  $H$ sending $g*x$ to the fundamental domain $X_H$.
  
  Now, denote by $|h|_{S_H}$ the length of $h$ in $H=\langle{S_H}  \rangle$.
  An equivalent manner to formulate the above \cref{Def:QuantMsgr} using cocycles
  is to replace $d_{S_H}(g* x,g \cdot x)$ in the integral by $|\alpha(g,x)|_{S_H}$.
\end{Rq}
\begin{Rq}[Quantitative measure equivalence]\label{Rq:MEQt}
  The authors also defined in \cite{DKLMT} the quantiative version of measure
  equivalence. Let $(\Omega,X_G,X_H,m)$ be a measure equivalence coupling from
  $G$ to $H$, and denote by $\alpha$ and $\beta$ the corresponding cocycles
  defined in \cref{eq:DefCocycle}.
  One says that the measure equivalence is
  \deffont{$(\varphi,\psi)$-integrable} if $\alpha$ and $\beta$ are
  respectively $\varphi$-integrable and $\psi$-integrable. This is equivalent to
  asking for the measure subgroup couplings $(\Omega,X_H,m)$ to be
  $\varphi$-integrable and $(\Omega,X_G,m)$ to be $\psi$-integrable.
\end{Rq}

Integrable measure subgroup couplings naturally arise from regular maps, and
thus, in particular, from coarse embeddings. This is what the following theorem
recalls. We refer to \bref{Ex:CoarseEmbeddings} for the definitions of coarse
embeddings and regular maps.
\begin{Th}[{\cite[Theorem~5.4]{DKLMT}}]\label{Th:RegularEmbeddings}
  Let $G$ and $H$ be two finitely generated groups.
  \begin{itemize}
  \item If there exists $k\in \bN$ such that there is an at most $k$-to-one
    $L^\infty$-integrable measure subgroup coupling from $G$ to $H$, then $G$
    regularly embeds into $H$.
  \item If $G$ is amenable and regularly embeds into $H$, then there exists
    $k\in \bN$ such that there is an at
    most $k$-to-one $L^\infty$-integrable measure subgroup coupling from $G$ to $H$.
  \end{itemize}
\end{Th}

\subsubsection{Isoperimetric profile}
Now recall from the introduction that the \deffont{isoperimetric profile} of
$G$ is the map $I_G$ defined by $ \profile_{G}(n):= \sup_{|A|\leq n}
{|A|}/{|\partial_{S_G} A|}$.
\begin{Th}[{\cite[Theorems 1 and 4.4]{DKLMT}}]\label{Th:ProfileetMEMSGR}
  Let $\varphi:\bR^*_+\rightarrow \bR^*_+$ be a non-decreasing map such that
  $t\mapsto t/\varphi(t)$ is also non-decreasing. Let $G$ and $H$ be two
  finitely generated groups.
  If either
  \begin{itemize}
  \item there exists a $(\varphi,L^0)$-measure equivalence coupling from $G$ to $H$;
  \item or there exists an at most $k$-to-one $\varphi$-integrable measure subgroup
    coupling from $G$ to $H$, for some $k\in \bN$; 
  \end{itemize}
  then $\varphi\circ I_H \preccurlyeq I_G$. 
\end{Th}

Assume for example that $H=\bZ$. Then its profile verifies $I_\bZ\simeq
\mathrm{id}$. The above theorem proves that, if there exists an
at most $k$-to-one $\varphi$-integrable measure sugroup coupling from a group
$G$ to $\bZ$, then $\varphi\preccurlyeq I_G$.

\subsection{Couplings building techniques}
In this section we recall the necessary material from \cite{DKLMT,CDKT}. We present the
tools needed to build the coupling of \bref{Th:CouplageLL}, namely sofic
approximations and Følner tiling sequences. The necessity of both these tools is
discussed in \bref{Rq:NecessiteDesOutils}.

\subsubsection{Sofic approximations}
In this paragraph $G$ will be a finitely generated group endowed
with a finite generating set $S_G$, and $(\mathcal{G}_n)_{n\in \bN}$ will be a
sequence of finite directed graphs, whose edges are labeled by the elements of
$S_G$. Furthemore “$x\in \calG_n$” will mean that $x$ is a vertex of $\calG_n$.

Let $r>0$ and denote by $\mathcal{G}_n^{(r)}$ the set of elements $x\in \mathcal{G}_n$ such that
$B_{\mathcal{G}_n}(x,r)$ is isomorphic to $B_G(e_G,r)$ seen as directed labeled
graphs.
We say that
$(\mathcal{G}_n)_{n\in \bN}$ is a \deffont{sofic approximation} if for every $r>0$
\begin{equation*}
  \lim_{n\rightarrow \infty}
  \frac{ \left\vert \mathcal{G}_n^{(r)}\right\vert}{\left\vert{\mathcal{G}_n}\right\vert}
  = 1.
\end{equation*}

\begin{Ex} Let $G$ be amenable and $(F_n)_n$ be a Følner sequence of
  $G$. For all $n\in \bN$, let $\calG_n$ be the graph whith vertex set $F_n$, and
  such that there is an edge labelled by $s\in S_G$ between two vertices $x$ and
  $y$ if and only if $x=ys$. Then $(\calG_n)_n$ is a Sofic approximation of $G$.
\end{Ex}

In \cite{CDKT} Delabie, Koivisto, Le Maître and Tessera prove a condition for a measure
subgroup coupling to be $\varphi$-integrable, using sofic approximations.

\begin{Th}[\cite{CDKT}]\label{Th:soficME}
  Let $\mathcal{U}$ be a non-principal ultrafilter on $\bN$. 
	Let $\varphi\colon \bR^+\rightarrow \bR^+$ be a non-decreasing map. Let $G$ and $H$
  be two finitely generated groups
  and let $(\mathcal{G}_n)_n$ be a sofic approximation of $G$.
  Let $\iota_n\colon \mathcal{G}_n\rightarrow
  H$ be an injective map such that,
  for every $s\in S_G$ there exists $\delta>0$ such that
    \begin{equation}
    \label{eq:CondsoficIntegrability}
      \lim_{R\to\infty}\sum_{r=0}^R \varphi\!\left(\delta r\right) \lim_{\mathcal{U}}
      \frac{\left\vert\left\{x\in \mathcal{G}_n^{(1)}\mid
        d_{H}(\iota_n(x),\iota_n(x\cdot s))=r\right\}\right\vert}%
  {\left\vert\mathcal{G}_n\right\vert}<\infty.
    \end{equation}
    Then there exists an at most one-to-one $\varphi$-integrable measure
    subgroup coupling from $G$ to~$H$.
\end{Th}

Given two amenable groups $G$ and $H$, and a sofic approximation $(\calG_n)_n$,
defining injections $\iota_n$ satisfying \cref{eq:CondsoficIntegrability} is not
always straightforward. Our strategy relies on the notion of \emph{Følner tiling
  sequence}, introduced in \cite{DKLMT}.  

\subsubsection{Følner tiling sequence: shifts and tiles}
Følner tiling sequences are tools to choose the needed sofic approximation
$(\calG_n)_n$ and define an injection $\iota_n$ from
$\calG_n$ to $H$, that respects the geometry, that is to say satisfies \cref{eq:CondsoficIntegrability}. 

\begin{Def}[{\cite{DKLMT}}]\label{Def:FOTS}\index{Følner tiling sequence}
  Let $G$ be an amenable group and $(\Sigma_n)_{n\in \bN}$ be a sequence of finite subsets
  of $G$. Define by induction the sequence $(T_n)_{n\in \bN}$ by $T_0:=\Sigma_0$ and
  $T_{n+1}:=\Sigma_{n+1}T_n$.\\
  We say that $(\Sigma_n)_{n\in \bN}$ is a (right) \deffont{Følner
    tiling sequence} if
  \begin{itemize}
  \item $(T_n)_{n\in \bN}$ is a (right) Følner sequence, \textit{viz.} $\lim_{n \rightarrow \infty}
      {|T_ng\backslash T_n|}/{|T_n|} = 0$ for all $g\in G$;
    \item $T_{n+1}= \sqcup_{\sigma \in \Sigma_{n+1}} \sigma T_n$.
  \end{itemize}
  We call $\Sigma_n$ a set of \deffont{shifts}\index{Shifts} and $T_n$
  a \deffont{tile}\index{Tiles}.
\end{Def}
Remark that in particular $T_n=\Sigma_n\cdots \Sigma_0$ for all $n\in \bN$, and for
all $g\in T_n$ there exists a unique sequence 
$(\sigma_i)_{i=0,\ldots,n}$ such that
\begin{equation*}
  g=\sigma_n\cdots \sigma_0 \quad \text{and} \quad 
  \sigma_i\in \Sigma_i,\ \forall i\in\{0,\ldots,n\}.
\end{equation*}
\begin{Ex}[{\cite[Proposition 6.10]{DKLMT}}]
  Let $G=\bZ$ and for all $n\in \bN$ let  $\Sigma_n:=\{0,2^{n}\}$. Then
  $(\Sigma_n)_n$ is a Følner tiling sequence and the corresponding sequence of
  tiles is the sequence of intervals of the form $T_n:=[0,2^{n+1}-1]$, for $n\in
  \bN$.
\end{Ex}

We refer to \bref{Sec:TilesBZ} for the definition of Følner tiling sequences in
diagonal products, and to \bref{Sec:SoficLamplighter} for a definition of
another tiling in a lamplighter group. Further examples of Følner tilings can be
found in \cite[Section~6]{DKLMT}.

\begin{Rq}[On the choice of tools]\label{Rq:NecessiteDesOutils}
  Our proof relies on the criterion given in \cref{Th:soficME} to build and
  quantify the coupling. This criterion requires us to build injections
  $\iota_n$ that respect the geometry of the groups, that is to say, satisfy
  \cref{eq:CondsoficIntegrability}. This is achieved by using Følner tiling
  sequences. 

  On the other hand, relying only on Følner tiling sequences is not possible for
  us in this case. More precisely, let $G$ and $H$ be two amenable groups with
  respective Følner tiling sequences
  $(\Sigma_n)_n$ and $(\Sigma^\prime_n)_n$. Proposition 6.9 in \cite{DKLMT}
  gives a criterion to show that $G$ and $H$ admit a $(\varphi,\varphi )$-integrable measure
  equivalence coupling. It
  requires that $\Sigma_n=\Sigma^\prime_n$, for all 
  $n\in \bN$.
  We did not manage to fulfill this strong requirement in our case, namely, when
  $G$ is a diagonal product and $H$ a lamplighter group. This is why, instead,
  we rely on \cref{Th:soficME}, which allows $\Sigma^\prime_n$ to
  contain more elements than $\Sigma_n$.
\end{Rq}

\section{Construction of the coupling}\label{Sec:Construction}
We now turn to the proof of \bref{Th:CouplageLL}.
In the following $\BZ$ will denote a diagonal product as defined in
\cref{Sec:DefdeBZ} and satisfying the hypotheses \textbf{(H)} on page~\cpageref{Hyp:H}.
In particular its isoperimetric profile is of the form $I_{\BZ}\simeq \rho \circ
\log$ for some $\rho \in \calC$. 
To prove \cref{Th:CouplageLL} we actually show
that the diagonal product obtained from the isoperimetric profile $\rho \circ
\log$ is the wanted group $G$. The integrability of the coupling is proved using
the criterion of \cref{Th:soficME}.

\subsection{Definition of the sofic approximations}\label{Sec:DefOfTheSA}
The purpose of this section is to define a sofic approximation $(\calG_n)_n$
of $G=\BZ$. It will be a Følner tiling sequence.
In order to define our family of injections
$\iota_n:\calG_n\rightarrow H$, we will also need a Følner tiling sequence in the
lamplighter group $H$. We thus start by
exhibiting Følner tiling sequences in both $G$ and $H$, and then extract
appropriate subsequences to work with.

\subsubsection{Følner tiling sequences in a diagonal product}\label{Sec:TilesBZ}
In \cite{EscalierOEZ} we constructed Følner tiling sequences of diagonal
products. 
We refer to \cref{Subsec:RangeDelta} for details on the range of an element.

\paragraph{Sequence of tiles} For all $n\in \bN$, let
\begin{equation}\label{eq:DefTn}
  T_n:=\left\{ (\mbff,t)\in \BZ \ | \ \range(\mbff,t)\in [0,\kappa^n-1] \right\}.
\end{equation}
The sequence $(T_n)_n$ is a right Følner sequence for $\BZ$
\cite[Proposition~2.13 and Lemma~3.10]{EscalierOEZ}. It verifies
$|\partial_{S_{\BZ}}T_n|/|T_n|\leq 2/\kappa^n$.

For all $n\in \bN$, let $\Landing(n)=\landing(\kappa^n-1)$,\label{Def:Landing} that is to say
$\Landing(n)$ is the unique integer such that $k_{\Landing(n)}\leq {\kappa}^{n}-1
<k_{\Landing(n)+1}$. For example if $k_n=\kappa^n$, then $\Landing(n)=n-1$. Note
that if
$(\mbff,t)$ belongs to $ T_n$, then for all $m> 
\Landing(n)$ it verifies $g^\prime_m=\neutre$.
We recall the following useful fact concerning the growth of the map $\Landing$.
\begin{Claim}[{\cite[Claim~3.8]{EscalierOEZ}}] \label{Claim:Landing}
  Let $n\geq 0$, then either $\Landing(n+1)=\Landing (n)$ or
  $\Landing(n+1)=\Landing (n)+1$. 
\end{Claim}

We computed in \cite[Lemma 4.2]{EscalierOEZ} the
value of $|T_n|$ for all $n\in \bN$, namely
\begin{equation}\label{eq:CardTn}
  |T_n|=\kappa^n \big( {|A| |B|} \big)^{\kappa^n}
    \prod^{\Landing(n)}_{m=1} {\left\vert {\Gamma^\prime}_m \right\vert}^{\kappa^{n}-k_m},
\end{equation}
and moreover showed \cite[Proposition 4.3]{EscalierOEZ} that there exist two
constants $\CTn, \CTnn>0$ depending only on $\BZ$, such that for all $n\in \bN$
\begin{equation}\label{eq:EncadrementLogTn}
    \CTn \kappa^{n-1}l_{\Landing(n)}\leq \ln |T_n| \leq \CTnn \kappa^nl_{\Landing(n)}.
\end{equation}
We now recall the definition of the corresponding shifts.

\paragraph{Sequence of shifts}\label{Def:SigmaZero}
Let $\Sigma_0:=T_0=\left\{ (\mbff,t)\in \BZ \ | \ \range(\mbff,t)= \{0\}
\right\}$. Now for all $n\geq 0$ 
we split $\Sigma_{n+1}$ in $\kappa$ parts defined as follows.
For all $j\in\{0,\ldots,\kappa-1\}$ we let $\Sigma^j_n$ be the set of
$(\mbfg,j\kappa^n)\in \BZ$ such that the following conditions are verified:
\label{Def:Sigman}
\begin{enumerate}
\item $\supp \left(g_0\right) \subseteq\left[ 0, j \kappa^n-1 \right]
  \cup \left[ (j+1)\kappa^n, \kappa^{n+1}-1 \right]$;
\item $\supp \left(g^\prime_m\right)\subseteq
  \left[k_m, j \kappa^n + k_m -1\right]\
  \cup \left[(j+1)\kappa^n, \kappa^{n+1}-1 \right]$ for all $m \in [1,\Landing (n)]$;
\item If $\Landing(n+1)=\Landing(n)+1$ then $\supp \left(g^\prime_{\Landing(n+1)}\right) 
  \subseteq\left[k_{\Landing(n+1)}, \kappa^{n+1}-1 \right]$;
\item $\supp \left(g^\prime_{m}\right) = \emptyset $ for all $ m \notin
  [0,\Landing (n+1)]$. 
\end{enumerate}
We refer to
\bref{fig:Sigmaj} for a representation of an element in such a part.
Finally, we define $\Sigma_{n+1}:= \cup^{\kappa -1}_{j=0}
\Sigma^j_{n+1}$. By \cite[Lemma 3.10]{EscalierOEZ}, this sequence verifies
$T_{n+1}=\Sigma_{n+1}T_n$ for all $n\in \bN$.

Let $(\mbfg,t)$ be an element of some $\Sigma^j_{n+1}$. We represent in
\cref{fig:Sigmaj} the supports and the sets where the maps $g_0, {g^\prime}_1,
\ldots, {g^\prime}_{\Landing (n+1)}$ take their 
values. The light-blue rectangle with dotted outline is in
$\Sigma^j_{n+1}$ if and only if $\Landing(n+1)=\Landing(n)+1$.
\begin{figure}[htbp]
  \centering
  \includegraphics[width=\textwidth]{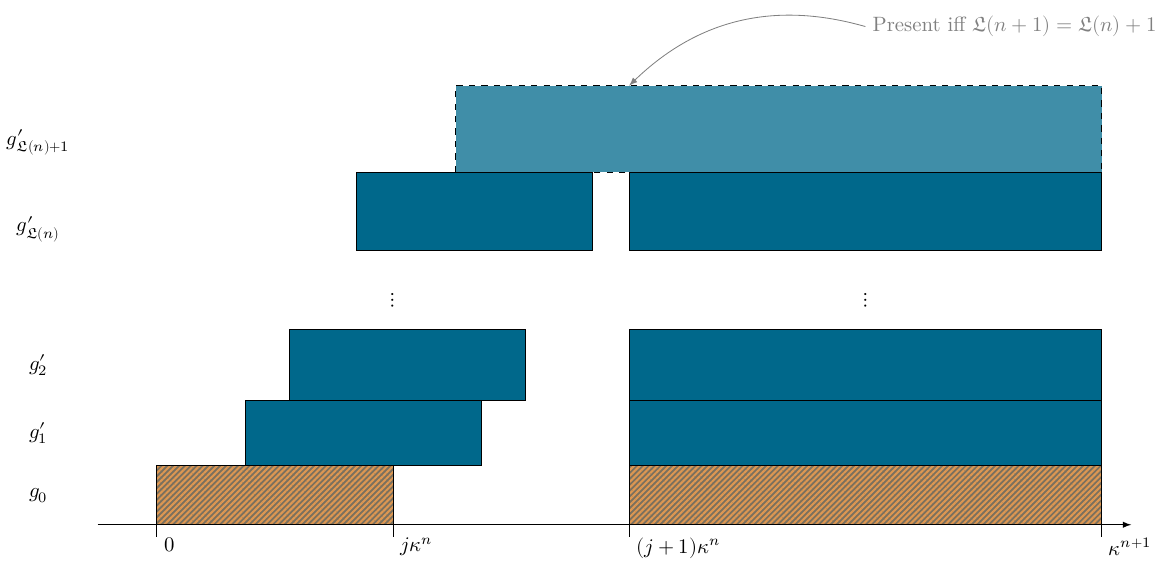}
  \caption{Support and values taken by $(\mbfg,t) \in \Sigma^j_{n}$, reprinted
    from {\cite{EscalierOEZ}}}
  \label{fig:Sigmaj}
\end{figure}

We conclude with some remarks on the cursor of elements of $\Sigma_n$. Recall
from \cpageref{Def:Range} that $\pi_2:\BZ\rightarrow \bZ$ 
denotes the map that sends an element of $\BZ$ to its cursor. 

\begin{Rq}[Cursor of tiling sequences]\label{Rq:CurseurShifts} 
  By definition of $\Sigma_0$ we have $\pi_2(\sigma)=0$ for all $\sigma\in
  \Sigma_0$. Now let $n\in \bN^*$. If
  $\sigma\in \Sigma_n$, then $\pi_2(\sigma)=j\kappa^n$, with $j\in
  \{0,\ldots,\kappa-1\}$.

  In particular, let $(\mbff,t)\in T_n$ and let $(\sigma_i)_{0\leq i\leq n}$ denote
  the unique sequence such that  $(\mbff,t)=\sigma_n\cdots \sigma_0$
  and $\sigma_i\in \Sigma_i$ for all $i\in\{0,\ldots,n\}$. Now decompose $t$ in
  base $\kappa$ as $t=\sum^{n-1}_{i=0} t_i\kappa^i$. Then for all $i\in
  \{0,\ldots,n-1\}$, we have $\pi_2(\sigma_{i+1})=t_i\kappa^i$.
\end{Rq}

\subsubsection{Tiling of the lamplighter group}\label{Sec:SoficLamplighter}
In the following we endow $H=(\bZ/m\bZ) \wr \bZ$ with the following finite, symmetric generating set:
$S_H:=\left\{ (\neutre,\pm 1)\right\}\cup \left\{(z\mathbf{1}_0,0) \, | \, z\in \bZ/m\bZ \right\}$.
Recall that a right Følner sequence of $H$ is given by
\begin{equation*}
  \FLL_n:= \left\{ \big( (\varepsilon_i)_i, t \big) \, \colon
    \, t \in [0,n-1], \, \supp \big({(\varepsilon_i)_{i}}\big) \subseteq [0,n-1] \right\}.
\end{equation*}

Our goal here is to extract a subsequence of $(\FLL_n)_{n\in \bN}$ to define a
\emph{tiling} for our group~$H$. So let $(d_n)_{n\in \bN}$ be a sequence of
integers and define $\SLL_0:=\FLL_{d_0}$. Now let $D_n:=\prod^n_{i=0} d_i$ and
consider for all $n\geq 0$
\begin{equation}\label{eq:DefShiftsLamplighter}
  \begin{split}
    \SLL_{n+1}:= \sqcup^{d_{n+1} -1}_{j=0}
    \left\{ \big( (\varepsilon_i)_i, j D_n  \big) \ | \ 
      \supp\big( (\varepsilon_i)_i\big) \right. \subseteq
      & \big[0,jD_n -1 \big] \\
    & \Big.
      \cup \big[(j+1)D_n,\ D_{n+1} -1 \big]
    \Big\}.
  \end{split}
\end{equation}
These sets will be the shifts of the wanted Følner tiling sequence. This is what
\bref{Lmm:PavageduLamplighter} formalises, but first let us give
some illustration of this tiling.
\begin{Ex}
  Assume that $d_0=2$ and $d_1=4$, then $D_0=2$ and $D_1=8$. Now consider
  $\big((\varepsilon_i)_i,t\big)\in \FLL_{D_0}$ and
  $\big((\varepsilon^{\prime}_i)_i,jD_0 \big)\in \SLL_1$. We represent the
  product of these two elements in \cref{fig:PavageLamp1} for $j=0$ and
  \cref{fig:PavageLamp2} for $j=2$. The dark blue
  squares correspond to lamp configurations coming from the element in $F_{D_0}$
  while the orange ones are coming from the shift. The cursor of this product,
  namely $t+jD_1$, belongs to the hatched blue rectangle. 
\end{Ex}
\begin{figure}[htbp]
  \centering
  \begin{subfigure}[b]{0.65\textwidth}
    \includegraphics[width=\textwidth]{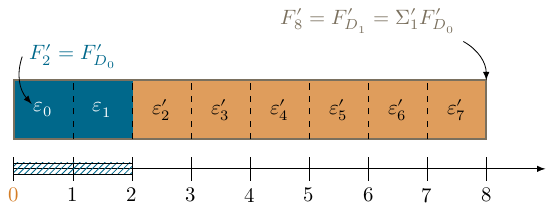} 
    \caption{Representation of $\big((\varepsilon^{\prime}_i)_i,0 \big)\big((\varepsilon_i)_i,t\big)$}
    \label{fig:PavageLamp1}
  \end{subfigure}

  \begin{subfigure}[b]{0.65\textwidth}
    \includegraphics[width=\textwidth]{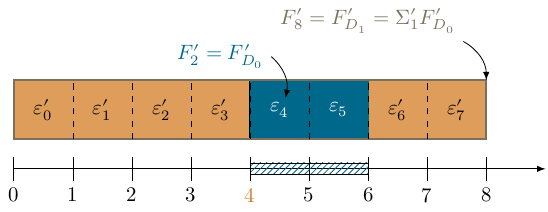} 
    \caption{Representation of $\big((\varepsilon^{\prime}_i)_i,2D_1\big)\big((\varepsilon_i)_i,t\big)$}
    \label{fig:PavageLamp2}
  \end{subfigure}
  \caption{Tiling of the lamplighter}
  \label{fig:PavageLamp}
\end{figure}

Let us now prove that it actually defines a Følner tiling sequence. 
\begin{Lmm}\label{Lmm:PavageduLamplighter} Let $(d_n)_{n\in \bN}$ be a sequence of
  positive integers and for all $n\in \bN$, let $D_n=\prod^n_{i=0}d_i$. Finally
  let $(\SLL_{n})_{n\in \bN}$ be as above.
  
  Then $(\SLL_n)_{n\in \bN}$
  is a right Følner tiling sequence and
  $\FLL_{D_{n+1}}=\SLL_{n+1}\FLL_{D_{n}}$, for all $n\in \bN$. 
\end{Lmm}
\begin{proof}
  Let $(d_n)_{n\in \bN}$  be a sequence of positive integers and for all $n\in
  \bN$ let $D_n:= \prod^{n}_{i=0}d_n$. Let $(\SLL_{n+1})$ be defined as in
  \cref{eq:DefShiftsLamplighter}.\\
  \textbf{Step 1} Let us prove that $\FLL_{D_{n+1}}=\SLL_{n+1}\FLL_{D_{n}}$, for
  all $n\in \bN$.
  \begin{itemize}
  \item We first show that for all $n\in \bN$, 
    $\SLL_{n+1}\FLL_{D_{n}}$ is contained in $\FLL_{D_{n+1}}$.\\
    Let $\big( (\varepsilon^\prime_i)_i, t \big)$ in $ \FLL_{D_n}$ and $j\in
    [0,d_{n+1}-1]$ and take $\big( (\varepsilon_i)_i, jD_n \big) \in 
    \SLL_{n+1}$. Then
    \begin{equation*}
      \big( (\varepsilon_i)_i, jD_n \big)\big( (\varepsilon^{\prime}_i)_i,t\big)
      = \left( \big(\varepsilon_i+ \varepsilon^{\prime}_{i-jD_n}\big)_{i},\  t+jD_n\right)
    \end{equation*}
    By the definition of $\SLL_{n+1}$ given in \cref{eq:DefShiftsLamplighter} and
    since $(\varepsilon^\prime_i)_i $ is supported on $[0,D_n-1]$, we have
    \begin{align*}
      \supp \left( \big(\varepsilon_i + \varepsilon^\prime_{i-jD_n} \big)_{i \in \bZ} \right)
      &\subseteq  \big[0,jD_n -1 \big] \cup \big[(j+1)D_n,\ D_{n+1} -1 \big]
        \cup [jD_n,(j+1)D_n-1],\\
      &= \left[ 0,D_{n+1}-1 \right].
    \end{align*}
    Finally, since $t\leq D_n-1$ and $j\leq d_{n+1}-1$, we get
    \begin{equation*}
      jD_n+t \leq (d_{n+1}-1)D_n + D_n-1 \leq D_nd_{n+1}-1=D_{n+1}-1.
    \end{equation*}
    Thus $\big( (\varepsilon_i)_i, jD_n \big)\big( (\varepsilon^{\prime}_i)_i,t\big)$
    belongs to $\FLL_{D_{n+1}}$.

  \item Now take $\big( (\omega_i)_i, t \big) \in \FLL_{D_{n+1}}$ and let us
    show that $\big( (\omega_i)_i, t \big) $ belongs
    to $\SLL_{n+1}\FLL_{D_{n}}$.\\
    First, remark that since $t\leq D_{n+1}-1$, and since $D_{n+1}=d_{n+1}D_n$, 
    there exists a unique $j \in [0,d_{n+1}-1]$ such that $jD_n\leq t\leq
    (j+1)D_n-1$.
    For such a $j$, let $t^{\prime}:=t-jD_n$ and let $(\varepsilon_i)_i$ and
    $(\varepsilon^\prime_i)_i$ be such that
    \begin{align*}
      \varepsilon_i&= 
                     \begin{cases}
                       \omega_i &\text{if} \ i\in [0,jD_n-1]\cup[(j+1)D_n,D_{n+1}-1],\\
                       \neutre &\text{else,}
                     \end{cases}\\
      \varepsilon^{\prime}_i&= 
                              \begin{cases}
                                \omega_{i+jD_n} &\text{if} \ i\in [0,D_n-1],\\
                                \neutre &\text{else}.
                              \end{cases}
    \end{align*}
    Then  $\big( (\varepsilon_i)_i, jD_n \big) \in \SLL_{n+1}$ and $\big(
    (\varepsilon^{\prime}_i)_i,t^{\prime}\big)\in \FLL_{D_n}$ and 
    $\big( (\varepsilon_i)_i, jD_n \big)\big(
    (\varepsilon^{\prime}_i)_i,t^{\prime}\big)$ is equal to $\big( (\omega_i)_i, t \big)$.
    Hence the equality of the lemma.
  \end{itemize}
  \textbf{Step 2} Let us prove that $(\SLL_n)_{n\in \bN}$ is a right Følner
  tiling sequence.
  \begin{itemize}
  \item The sequence of tiles $\big(\FLL_{D_n}\big)_n$ is a subsequence of the right Følner
    sequence $\big(\FLL_n\big)_n$ of~$H$. Therefore, it is itself a right Følner sequence.
  \item We now show that
  $\sigma\FLL_{D_{n}}\cap\sigma^{\prime}\FLL_{D_n}=\emptyset$ for all $\sigma
  \neq \sigma^{\prime} \in \SLL_{n+1}$.\\
  So take $\big( (\varepsilon_i)_i, jD_n \big)$ and $\big(
  (\varepsilon^{\prime}_i)_i, j^{\prime}D_n \big)$ in $\SLL_{n+1}$ and let
  $((\omega_i)_i,t)$ and $ ((\omega^{\prime}_i)_i,t^{\prime})$ in $\FLL_{D_n}$. If 
  \begin{equation}\label{eq:lienentreleselementsFOTSLamp}
    \big( (\varepsilon_i)_i, jD_n \big)((\omega_i)_i,t) 
    = \big( (\varepsilon^{\prime}_i)_i, j^{\prime}D_n \big)((\omega^{\prime}_i)_i,t^{\prime}),
  \end{equation}
  then in particular $t+jD_n=t^{\prime}+j^{\prime}D_n$. But $t,t^{\prime}<D_n$
  thus the last equality implies $t=t^{\prime}$ and therefore $j=j^{\prime}$. In
  particular $(\varepsilon_i)_i$ and $(\varepsilon^{\prime}_i)_i$ are supported
  on the same set, namely $[0,jD_n-1]\cup[(j+1)D_n,D_{n+1}-1]$. This
  last set is disjoint from  $[jD_n,(j+1)D_n-1]$ which is the interval where
  $(\omega_{i-jD_n})_i$ and $(\omega_{i-jD_n})_i$ are supported. Combining
  this with \cref{eq:lienentreleselementsFOTSLamp} we thus get that
  $\varepsilon_i=\varepsilon^{\prime}_i$ for all $i$. Hence the result.
  \qedhere
\end{itemize}
\end{proof}

We conclude with the following lemma.
\begin{Lmm}\label{Lmm:CardLamplighter}
  Let $(d_n)_{n\in \bN}$ be a sequence of positive integers and for all $n\in
  \bN$, let $D_n=\prod^n_{i=0}d_i$.  Let $(\SLL_{n})_{n\in \bN}$ be as in
  \bref{eq:DefShiftsLamplighter}. Then, the sequence of tiles $\big(\FLL_{D_n}\big)_n$
  verifies $ \mathrm{diam}\left( \FLL_{D_n} \right)\leq 3D_n$ for all $n\in \bN$.
  Moreover, for all $n\in \bN$ we have
  \begin{equation*}
    \left\vert \FLL_{D_n} \right\vert = D_n m^{D_n}
    \quad \text{and} \quad
    \left\vert \SLL_{n+1} \right\vert = d_{n+1}m^{(d_{n+1}-1)D_n}.
  \end{equation*}
\end{Lmm}
\begin{proof}
  By choice of the generating set $S_H$, the bound on the diameter is immediate.
  
  Let $n\in \bN$ and consider $\big( (\varepsilon_i)_i, t\big)\in \FLL_{D_n}$. 
  By definition of $\FLL_{D_n}$ (see \cpageref{Sec:SoficLamplighter}),
  there are exactly $D_n$ possible values for the cursor $t$, and $m^{D_n}$
  possible values for the sequence $(\varepsilon_i)_i$. Therefore $\FLL_{D_n}$
  contains exactly $ D_n m^{D_n}$ elements. Furthermore, since
  $\FLL_{D_{n+1}}=\SLL_{n+1}\FLL_{D_n}$, the shift $ \SLL_{n+1}$ contains
  therefore $d_{n+1}m^{(d_{n+1}-1)D_n}$ elements.
\end{proof}

We thus know how to build Følner tiling sequences for $H$. To be able to define
$\iota_n$, we want $\Sigma_n$ to embed in $\SLL_n$ for all $n$. We thus now have to
specify the sequence $(d_n)_{n\in \bN}$.
\subsubsection{Defining the sofic approximations}\label{subsec:DefSoficApprox}
In the following we denote by $(\Sigma_n)_n$ the Følner tiling sequence of $\BZ$
defined in \bref{Sec:TilesBZ} and by $(T_n)_n$ the corresponding tiling of $\BZ$
defined in \cref{eq:DefTn}. In particular it verifies $T_{n+1}=\Sigma_{n+1}T_n$ for
all $n\in \bN$. 

Given a sequence $(d_n)_n$ of integers we also let $D_n:=\prod^n_{i=0}d_i$ and
denote by $(\SLL_n)_n$ the corresponding Folner tiling sequences defined in
\cref{eq:DefShiftsLamplighter}. The goal is to define $(d_n)_n$ such that
one can embed $\Sigma_n$ in $\SLL_n$ for all $n$.

So for all $n\in\bN$, let $\calG_n:=T_n$, and now let us define inductively the
subsequence $\big(\calH_n\big)_n$ of $\big(\FLL_n \big)_n$. 
First, let $d_0:=\min\{ j: |T_0|\leq |\FLL_{j}|\}$ and $\calH_0:=\FLL_{d_0}$. By
definition of the Følner tiling sequences we have $\FLL_{d_0}=\SLL_{0}$ and
$\Sigma_0=T_0$, thus $|\Sigma_0|\leq |\SLL_0|$. Moreover, by definition of
$d_0$ and \bref{Lmm:CardLamplighter}, we have
\begin{equation}\label{eq:CardSigmaZero}
  (d_0-1)m^{d_0-1}<
  \left\vert \Sigma_{0} \right\vert
  \leq d_0m^{d_0}.
\end{equation}

Then let $n\geq 0$ and assume that we have defined the sequence $(d_i)_{0\leq i
  \leq n}$ and $\calH_n=\FLL_{D_n}$. Let $d_{n+1}$ be the minimal integer
such that the set $\SLL_{n+1}$ defined in \cref{eq:DefShiftsLamplighter}
contains at least $\left\vert{\hS_{n+1}}\right\vert$ elements, \emph{viz.}
\begin{equation}\label{eq:Relationdn}
  \left( d_{n+1}-1 \right) m^{D_n(d_{n+1}-2)} <
  \left\vert{\hS_{n+1}}\right\vert \leq
  d_{n+1} m^{D_n(d_{n+1}-1)}.
\end{equation}
Remark that in
particular, one can embed $\hS_{n+1}$ in $\SLL_{n+1}$. Finally let
$\calH_{n+1}=\FLL_{D_{n+1}}$.\label{Def:CalHn} It defines by induction a sequence
$(\calH_n)_{n\in \bN}$.
We refer to \cref{tab:Recap} for a
summary of the different objects defined above.

\begin{table}[htbp]
  \centering
  \begin{tabular}{ccc}
    \toprule
    $\calG_0=\Sigma_0=T_0$&~& $\calG_{n+1}=T_{n+1}=\Sigma_{n+1}T_n$\\
    \midrule
    $d_0:=\min\{ j: |T_0|\leq |\FLL_{d_0}|\}$ & ~&$\calH_0:=\FLL_{d_0}=\SLL_0$\\
    \midrule
    $d_{n+1}$ defined by \bref{eq:Relationdn} & ~&$D_{n+1}=d_{n+1}D_n=\prod^{n+1}_{i=0}d_i$\\
    $\SLL_{n+1}$ defined by \bref{eq:DefShiftsLamplighter}& ~&$\calH_{n+1}=\FLL_{D_{n+1}}=\SLL_{n+1}\FLL_{n+1}$\\
    \bottomrule
  \end{tabular}
  \caption{Definition of the sofic approximations}
  \label{tab:Recap}
\end{table}

\subsection{Quantification}\label{Sec:Qt}
The purpose if this section is to fulfil the criterion given by
\bref{Th:soficME}.

\subsubsection{Definition of the injection} 
Let us define, for all $n\in \bN$, the injection $\iota_n$ that embeds $\calG_n$ in $\calH_n$.

First remark that, by definition of $(\SLL_n)_n$, we have
$|\Sigma_n|\leq |\SLL_n|$, for all $n\in \bN$. 
There thus exists an injection
$\nu_n$\label{nun} from $\hS_n$ to $\SLL_n$. From now on, fix an arbitrarily
chosen sequence of injections $(\nu_n)_n$.  

Now let $n\in \bN$. Since $\big(\hS_{i}\big)_{i\in \bN}$ is a Følner tiling sequence,
one can write every element of $\calG_{n}$ as a product $\sigma_{n} \cdots \sigma_0$
where $\sigma_i \in\hS_i$ is uniquely determined for all $i$. We can thus
define without ambiguity the following map $\iota_n$.

\begin{Lmm}\label{Lmm:DefIotan} Let $n\in \bN$. The map defined by 
  \begin{equation*}
    \iota_{n} \ : \
    \begin{cases}
      \calG_n = T_{n}
      &\rightarrow \calH_n=F^\prime_{D_n},\\
      \sigma_n\cdots \sigma_0
      & \mapsto
        \nu_n\left( \sigma_n\right) \cdots\nu_0\left( \sigma_0\right).
    \end{cases}
  \end{equation*}
  where $\sigma_i\in\Sigma_i$ for all $i\leq n$, is an injection
  from $\calG_n$ to $\calH_n\subseteq H$.
\end{Lmm}

\begin{proof} Let $n\in \bN$. 
  The map $\iota_n$ thus defined does not depend on the choice of the
  decomposition of an element of $T_n$ in a product of shifts. Indeed, by the
  preceding discussion, there is only one such possible decomposition.

  Now let
  $x,x^{\prime}\in \calG_n$. .
  Let $x=\sigma_n \cdots \sigma_0$ and
  $x^\prime=\sigma^\prime_n\cdots\sigma^\prime_0$ be the decompositions in
  product of shifts of $x$ and $x^\prime$. In particular, for all  $i\in
  \{0,\ldots,n\}$ the shifts $\sigma_i,\sigma^{\prime}_i$ are elements of $\hS_{i}$.
  Then by definition of $\iota_n$ we have  $\iota_n(x)=\prod^n_{i=0} 
  \nu_i\left( \sigma_i \right)$ and $\iota_n(x^{\prime})=  \prod^n_{i=0} \nu_i\left( \sigma^{\prime}_i
  \right)$. But $\nu_i(\sigma_i)$ and $\nu_i(\sigma^{\prime}_i)$ belong to $\SLL_i$ for all $i$, thus
  $\prod^n_{i=0} \nu_i\left( \sigma_i \right)$ is the decomposition of
  $\iota_n(x)$ in product of shifts and $\prod^n_{i=0} \nu_i\left( \sigma^\prime_i
  \right)$ the one of $\iota_n(x^{\prime})$ . Since $(\SLL_{n})_n$ is a Følner
  tiling shift, this decomposition is unique. Thus if
  $\iota_n(x)=\iota_n(x^{\prime})$ then
  $\nu_i(\sigma_i)=\nu_{i}(\sigma^{\prime}_i)$ for all $i$. Hence
  $\sigma_i=\sigma^{\prime}_i$ since $\nu_i$ is an injection for all $i$, and
  therefore $x=x^{\prime}$. Hence the injectivity of $\iota_n$.
\end{proof}

\subsubsection{Distance}
In order to verify \bref{eq:CondsoficIntegrability}, we now need to estimate the
values taken by the distance between  $\iota_n\big((\mbff,t)\big)$ and
$\iota_n\big((\mbff,t)s\big)$ when $(\mbff,t)\in \BZ$ and $s\in
\mathcal{S}_{\BZ}$. 
We distinguish two
cases depending on whether $s=(\mbfe,1)$ or not. But first let us introduce some
notations.

Recall that $\calG^{(1)}_n$ is defined as
\begin{equation*}
  \calG^{(1)}_n=\left\{ x \in \mathcal{G}_n \ | \ 
    B_{\mathcal{G}_n}(x,1) \ \text{is isometric to}\ B_G(e_G,r)\right\}.
\end{equation*}
Let $t$ in $\{0,\ldots,\kappa^n-1\}$ and let $t=\sum^{n-1}_{i=0} t_i\kappa^i$ be
the decomposition in base $\kappa$ of $t$. If $(\mbff,t)$ belongs to
$\calG^{(1)}_n$ then $t<\kappa^n-1$ and thus there exists $i\in \{0,\ldots,n-1\}$ such
that $t_i<\kappa-1$. Therefore, we can define
\begin{equation}
  \label{eq:Defizerot}
  i_0(t):=\min\{i\leq n \ | \ t_i<\kappa-1 \}.
\end{equation}
This index corresponds to the coefficient $t_i$ that will
absorb the carry when we add one to~$t$. In other words, the decomposition
of~$t+1$ in base $\kappa$ is given by~$t+1=(t_{i_0(t)}+1)\kappa^{i_0(t)} +
\sum^{n-1}_{i=i_0(t)} t_i \kappa^i$.

\begin{Lmm}\label{Lmm:ActionOnTheShifts}
  Let $n\in \bN$ and $s\in \mathcal{S}_{\BZ}$. For all $t\in
  \{0,\ldots,\kappa^n-1\}$, let $i_0(t)$ be as in \cref{eq:Defizerot}.
  Then for all $(\mbff,t)\in \calG^{(1)}_n$ we have 
  \begin{equation*}
    d_H\left(\iota_n\big((\mbff,t)\big),\iota_n\big((\mbff,t)s\big)
    \right)\leq
    \begin{cases}
      3D_0 & \text{if} \ s\neq (\mbfe,1),\\
      3D_{i_0(t)}                 & \text{if} \ s=(\mbfe,1).
    \end{cases}
  \end{equation*}
\end{Lmm}
The following proof relies on the definition of the range given in
\cref{Def:Range}, \cpageref{Def:Range}. 
\begin{proof}
  Recall that
  $\calG_n=T_n=\Sigma_n\cdots \Sigma_0$ and since $(\mbff,t)\in T_n$, there exists a
  unique sequence $(\sigma_i)_{0\leq i\leq n}$ such that $(\mbff,t)=\sigma_n\cdots
  \sigma_0$ and $\sigma_i\in \Sigma_i$ for all $i\in\{0,\ldots,n\}$.
  Similarly, we denote by $(\sigma^\prime_n)_{0\leq i\leq n}$ the unique
  sequence such that $(\mbff,t) s=\sigma^\prime_n\cdots
  \sigma^\prime_0$ and $\sigma^\prime_i\in \Sigma_i$ for all $i$.
  \smallskip
  
  \noindent\textbf{First case} Let $s\in \mathcal{S}_{\BZ} \backslash \{(\neutre,1)\}$.

  \begin{itemize}
  \item Let us first prove that $\sigma_0 s$ belongs to $T_0$.\\
    Since $\sigma_0\in \Sigma_0=T_0$ it verifies $\range(\sigma_0)=\{0\}$.
    Moreover, since $s$ belongs to $\mathcal{S}_{\BZ}\backslash\{(\neutre,1)\}$ its cursor is
    equal to $0$. Therefore, by definition of the range, we have
    $\range(\sigma_0 s)=\{0\}$. Hence $\sigma_0 s\in T_0$. 
  \item Let us now show that $\sigma_i=\sigma^\prime_i$ for all $i>0$.\\
    Remark that
    \begin{equation*}
      (\mbff,t) s= \sigma_n\cdots \sigma_0 s
      = \sigma_n\cdots \sigma_1 \left(\sigma_0 s\right).
    \end{equation*}
    But by the last point $\sigma_0 s$ belongs to $\Sigma_0$, thus the
    above equality gives a decomposition of $(\mbff,t)s$ in a product of shifts.
    By uniqueness of this decomposition, we thus have
    $\sigma^\prime_0=\sigma_0s$ and $\sigma_i=\sigma^\prime_i$ for all
    $i>0$.
  \item Let us now bound by above the distance between
    $\iota_n\big((\mbff,t)\big) $ and $ \iota_n\big((\mbff,t)s\big)$.\\
    Using the definition of $\iota_n$ given by \bref{Lmm:DefIotan}, then
    the above point, and finally \bref{Lmm:CardLamplighter}, we obtain that
    \begin{align*}
      d_H\left(\iota_n\big((\mbff,t)\big),\iota_n\big((\mbff,t)s\big)\right)
      &=
      d_H\big(\nu_0\left(\sigma_0 \right), \nu_0\left( \sigma_0 s \right) \big),\\
      &\leq \mathrm{diam}(\calH_0)=3D_0.
    \end{align*}
  \end{itemize}
  \noindent\textbf{Second case} Now assume that $s=(\mbfe,1)$.
  \begin{itemize}
  \item Let us prove that $\sigma_{i_0(t)}\cdots \sigma_0(\neutre,1)$ belongs to
    $T_{i_0(t)}$.\\
    First recall from \cpageref{Def:Range} that $\pi_2:\BZ\rightarrow \bZ$
    denotes the map that sends an element of $\BZ$ to its cursor. Using 
    \bref{Rq:CurseurShifts} we get that
    $\pi_2(\sigma_0)=0$ and 
    $\pi_2(\sigma_{i+1})=t_i\kappa^{i}$ for all $i\in \{0,\ldots,n-1\}$. Therefore
    \begin{equation*}
      \pi_2\left( \sigma_{i_0(t)}\cdots \sigma_0 \right)
      =\sum^{i_0(t)}_{i=0}t_i\kappa^{i}.
    \end{equation*}
    But, by definition of $i_0(t)$, we have $t_i=\kappa-1$ for all
    $i<i_0(t)$ and $t_{i_0(t)}<\kappa-1$, thus
    $\pi_2\left( \sigma_{i_0(t)}\cdots \sigma_0 \right)<\kappa^{i_0(t)}-1$. Hence
    \begin{equation*}
      \pi_2\left( \sigma_{i_0(t)}\cdots \sigma_0 (\neutre,1)\right)
      =\pi_2\left( \sigma_{i_0(t)}\cdots \sigma_0 \right) + 1
      \leq \kappa^{i_0(t)}-1.
    \end{equation*}
    Since $\sigma_{i_0(t)}\cdots \sigma_0$ belongs to $T_{i_0(t)}$, its range is
    included in $[0,\kappa^{i_0(t)}-1]$. Thus, using the above inequality we get
    \begin{align*}
      \range\left(\sigma_{i_0(t)}\cdots \sigma_0 (\neutre,1) \right)
      &\subseteq \range\left(\sigma_{i_0(t)}\cdots \sigma_0\right)
      \cup \left\{\pi_2\left( \sigma_{i_0(t)}\cdots \sigma_0 (\neutre,1)\right)\right\}\\
      &\subseteq [0,\kappa^{i_0(t)}-1],
    \end{align*}
    namely $\sigma_{i_0(t)}\cdots \sigma_0 (\neutre,1)$ belongs to
    $T_{i_0(t)}$.
  \item Let us now show that $\sigma_i=\sigma^\prime_i$ for all $i>i_0(t)$.\\
    By the above point there exists a sequence $(\bar{\sigma}_i)_{0\leq i\leq
      i_0(t)}$ such that 
    $\bar{\sigma}_i \in \Sigma_i$ for all $i\leq i_0(t)$ and such that
    $\sigma_{i_0(t)}\cdots \sigma_0 s=\bar{\sigma}_{i_0(t)}\cdots
    \bar{\sigma}_0$.
    Therefore
    \begin{align*}
      (\mbff,t) s
      = \sigma_n\cdots \sigma_0  s
      &= \sigma_n\cdots \sigma_{i_0(t)+1}
        \left(\sigma_{i_0(t)}\cdots \sigma_0  s\right),\\
      &= \sigma_n\cdots \sigma_{i_0(t)+1}
        \left(\bar{\sigma}_{i_0(t)}\cdots \bar{\sigma}_0\right),
    \end{align*}
    where $\bar{\sigma}_i$ belongs to $\Sigma_i$ for all $i\leq i_0(t)$ and
    $\sigma_i$ belongs to $\Sigma_i$ for all $i>i_0(t)$. Hence the
    above equality gives a decomposition of $(\mbff,t)s$ in a product of shifts.
    By uniqueness of this decomposition, 
    we thus have
    $\sigma^\prime_i=\bar{\sigma}_i$ for all $i\leq i_0(t)$ and
    $\sigma_i=\sigma^\prime_i$ for all $i>i_0(t)$.
  \item Using the definition of $\iota_n$ given by \bref{Lmm:DefIotan}, then
    the above point, and finally \bref{Lmm:CardLamplighter}, we obtain that 
    \begin{align*}
      d\left(\iota_n\big((\mbff,t)\big),\iota_n\big((\mbff,t) s\big)\right)
      &\leq \mathrm{diam}(\calH_{i_0(t)})=3D_{i_0(t)}.\qed 
        \phantom{\qedhere}
    \end{align*}
  \end{itemize}
\end{proof}
We conclude with an estimate of the proportion of elements in $\calG_n$ such that
$i_0(t)=i$ for some fixed $i\in \{0,\ldots,n-1\}$.
\begin{Lmm}\label{Lmm:Enumeration}
  Let $n\in \bN$ and $i_0(t)$ be as in \cref{eq:Defizerot}. Then, for
  all $i\in \{0,\ldots, n-1\}$, we have
  \begin{equation*}
    \left\vert\left\{ (\mbff,t)\in \calG_n \ : \ i_0(t)=i \right\} \right\vert
    = \frac{(\kappa-1)}{\kappa^i} \left\vert\calG_n\right\vert.
  \end{equation*}
\end{Lmm}
\begin{proof} So let $i\in \{0,\ldots,n-1\}$. 
  As above, let $t=\sum^{n-1}_{j=0}t_j\kappa^j$ be the decomposition in base
  $\kappa$ of $t$. For a given $j\in \{0,\ldots,n-1\}$, the digit $t_j$ can
  uniformly take $\kappa$ different values, namely any value between zero and $\kappa-1$.
  Thus, the proportion of $(\mbff,t)\in \calG_n$ such that $t_{i}<\kappa-1$
  and $t_{j}=\kappa-1$ for all $j<i$ is equal to $(\kappa-1)\kappa^{-i}$. Hence
  the lemma.
\end{proof}
\subsubsection{Integrability}

We now turn to the proof of the quantification and first bound by above the value of~$D_n$.
\begin{Lmm}\label{Claim:Dn}
  There exists $\CSnn\geq 1$ depending only on $\BZ$, such that $D_{n}
  \leq \CSnn \kappa^nl_{\Landing(n)}$, for all $n\in \bN$.
\end{Lmm}

\begin{proof}
  We start by proving some inequalities and then show the upper bound by induction on~$n$.
  \smallskip

  \noindent\noindent\textbf{Useful remarks} 
  \begin{enumerate}
  \item\label{item:uiun} First, note that since $m\geq 2$ we have $\ln(m)\geq 1/2$ and therefore
    $1/\ln(m)\leq 2$.
  \item\label{item:uideux} By definition of $\Sigma_0$, and \cref{H:1} of hypotheses
    \textbf{(H)}, we have $|\Sigma_0|=q=12$. Therefore
    $\ln|\Sigma_0|=\ln(q)\geq 1$.
  \item\label{item:uitrois} Let us show that $d_{n+1}\geq 2$ for all $n\in \bN$.\\
    Let $n\in \bN$ and assume towards a contradiction that $d_{n+1}=1$. By the
    right most inequality of \bref{eq:Relationdn}, this implies
    $|\Sigma_{n+1}|\leq m^{d_0\cdots d_n \cdot 0}=1$. But the value of $|T_i|$ is
    strictly increasing. Since $T_{n+1}=\Sigma_nT_n$, this forces $\Sigma_{n+1}$
    to contain at least $2$ elements. Hence the contradiction.
  \item \label{item:uiquatre} Let us now prove that, for all $n\in \bN^*$ we
    have $\ln \left\vert\Sigma_{n} \right\vert \leq 
      \CTnn \kappa^{n}l_{\Landing(n)}$.\\
      Let $n\in \bN$.
      Recall that 
      $T_{n+1}=\Sigma_{n+1}T_n$. In particular $|\Sigma_{n+1}|=|T_{n+1}|/|T_n|$.
      Combining this with the upper bound in \bref{eq:EncadrementLogTn}, leads to
      \begin{align*}
        \ln \left\vert\Sigma_{n+1} \right\vert=
        \ln \left\vert T_{n+1}\right\vert-\ln \left\vert T_{n}\right\vert
        \leq \ln \left\vert T_{n+1}\right\vert
        &\leq \CTnn \kappa^{n+1}l_{\Landing(n+1)}.
      \end{align*}
  \end{enumerate}
  Let $\CSnn:=\max \left\{3\ln(q),6\CTnn \right\}$. Remark in particular
  that $\CSnn\geq 1$, by the above \cref{item:uideux}. We now aim to
  show that $D_n\leq\CSnn \kappa^nl_{\Landing(n)}$, for all $n\in \bN$
  \smallskip

  \noindent\textbf{Base case} Let us first treat the case where $n=0$.
  Remark that by \bref{eq:CardSigmaZero} we
  have $(d_0-1)m^{d_0-1}\leq |\Sigma_0|$. Hence $(d_0-1)\ln(m)\leq \ln\left|
    \Sigma_0 \right|$.
  Therefore, using first the latter
  inequality, then that $\ln |\Sigma_0| \geq 1$ (\cref{item:uideux}),
  then that $1/\ln(m)\leq 2$ (\cref{item:uiun}), leads to
  \begin{equation*}
    d_0\leq \frac{\ln|\Sigma_0|}{\ln(m)} +1 \leq
    \ln|\Sigma_0|\left(\frac{1}{\ln(m)} +1  \right)\leq 3  \ln|\Sigma_0|.
  \end{equation*}
  But by \cref{item:uideux}, $|\Sigma_0|=q$. Recalling moreover that $D_0=d_0$ thus gives
  $D_0\leq 3 \ln|\Sigma_0|=3\ln(q)$. Furthermore, using the definition of
  $\Landing$ given \cpageref{Def:Landing}
  and using \cref{H:2} of hypotheses \textbf{(H)} leads to $l_{\Landing(0)}\geq l_0=1$. The
  wanted inequality then comes noting that $\CSnn\geq 3\ln(q)$.
  \smallskip
  
  \noindent\textbf{Induction} Now let $n\in \bN$ and assume that $D_n\leq
  \CSnn \kappa^nl_{\Landing(n)}$.
  We showed in \cref{item:uitrois} that $d_{n+1}\geq 2$. We distinguish two
  cases depending on whether $d_{n+1}=2$ or $d_{n+1}\geq 3$.
  \begin{itemize}
  \item If $d_{n+1}=2$, then $D_{n+1}=2D_n$. Using our assumption on $D_n$,
    and that $\kappa\geq 3$, we obtain 
    \begin{align*}
      D_{n+1}=2D_n\leq 2\cdot \CSnn \kappa^nl_{\Landing(n)}
      &\leq \CSnn \kappa^{n+1}l_{\Landing(n)}.
    \end{align*}
    Recall from \bref{Claim:Landing} that $\Landing(n)\leq
    \Landing(n+1)$. Furthermore, by \cref{H:4} of \textbf{(H)} we know that
    $(l_i)_i$ is a subsequence of a geometric sequence. In particular it is
    non-decreasing. Thus $l_{\Landing(n)}\leq l_{\Landing(n+1)}$ and hence the
    wanted inequality..
  \item Assume now that $d_{n+1}\geq 3$. Taking the logarithm of the left most
    inequality of \bref{eq:Relationdn} then leads to $D_n(d_{n+1}-2)\ln(m)\leq
    \ln|\Sigma_{n+1}|$. Since $D_{n+1}=d_{n+1}D_n$ we deduce
    \begin{equation*}
      D_{n+1}\leq \frac{\ln|\Sigma_{n+1}|}{\ln(m)} \frac{d_{n+1}}{(d_{n+1}-2)}.
    \end{equation*}
    But since $d_{n+1}\geq 3$, we have $d_{n+1}/(d_{n+1}-2)\leq 3$. Recalling
    that $1/\ln(m)\leq 2$ from the above useful remarks thus
    leads to $D_{n+1}\leq 6 {\ln|\Sigma_{n+1}|}$. Then applying
    \cref{item:uiquatre} of the aforementioned remarks, and using that $\CSnn\geq 6\CTnn$, implies
    \begin{align*}
      D_{n+1}&\leq 6 \CTnn \kappa^{n+1}l_{\Landing(n+1)}\leq \CSnn \kappa^{n+1}l_{\Landing(n+1)}.
               \qed      
               \phantom{\qedhere}
    \end{align*}
  \end{itemize}
\end{proof}

We now turn to the proof of our main result.

\begin{proof}[{Proof of \cref{Th:CouplageLL}}] So let $\rho \in \calC$, let
  $\varepsilon>0$ and define $\varphi_\varepsilon={\rho}^{1-\varepsilon}$. Consider
  $\BZ$ fulfilling the conditions \textbf{(H)}
  and such that $I_{\BZ}\simeq \rho \circ \log$. Let $G:=\BZ$
  and $H:=(\bZ/m\bZ)\wr \bZ$. Let $(\calG_n)_{n}$ and $(\cH_n)_{n}$ be as
  defined in \bref{subsec:DefSoficApprox}.

  Consider $s \in \mathcal{S}_{\BZ}$ and fix $R>0$ and $n\in \bN$. First remark
  that:
  \begin{align*}
    \sum_{r=0}^R \varphi_\varepsilon
    &\left(r\right) \lim_{\mathcal{U}}
    \frac{\left\vert\left\{x\in \mathcal{G}_n^{(1)}\mid
    d_{H}(\iota_n(x),\iota_n(x\cdot s))=r\right\}\right\vert}%
    {\left\vert\mathcal{G}_n\right\vert}\\
    &\leq \sum_{i\geq 0} \varphi_\varepsilon(3D_i) \sum^{3D_i}_{r=3D_{i-1}+1}  \lim_{\mathcal{U}}
      \frac{\left\vert\left\{x\in \mathcal{G}_n^{(1)}\mid
      d_{H}(\iota_n(x),\iota_n(x\cdot s))=r\right\}\right\vert}%
      {\left\vert\mathcal{G}_n\right\vert},\\
    &\leq \sum_{i\geq 0} \varphi_\varepsilon(3D_i) \lim_{\mathcal{U}}
      \sum^{3D_i}_{r=3D_{i-1}+1} 
      \frac{\left\vert\left\{x\in \mathcal{G}_n^{(1)}\mid
      d_{H}(\iota_n(x),\iota_n(x\cdot s))=r\right\}\right\vert}%
      {\left\vert\mathcal{G}_n\right\vert},\\
    &\leq \sum_{i\geq 0} \varphi_\varepsilon(3D_i) \lim_{\mathcal{U}}
      \frac{\left\vert\left\{x\in \mathcal{G}_n^{(1)}\mid
      d_{H}(\iota_n(x),\iota_n(x\cdot s))\in\left[ 3D_{i-1}+1,3D_i \right]\right\}\right\vert}%
      {\left\vert\mathcal{G}_n\right\vert}.
  \end{align*}
  We distinguish two cases depending on whether $s=(\neutre,1)$ or not.
  \medskip
  
  \noindent\textbf{First case} 
  Let $s\in \mathcal{S}_{\BZ}\backslash\{\mbfe,1\}$. By \bref{Lmm:ActionOnTheShifts}
  the distance between $\iota_n(x)$ and $\iota_n\left( x\cdot s
  \right)$ is bounded on $\calG_n$ by $3D_0$. Therefore, the above inequalities
  lead to
  \begin{equation*}
    \sum_{r=0}^R \varphi_\varepsilon\left(r\right)\lim_{\mathcal{U}}
    \frac{\left\vert\left\{x\in \mathcal{G}_n^{(1)}\mid
    d_H \left(\iota_n(x),\iota_n\left(x\cdot s\right)\right)=r\right\}\right\vert}%
{\left\vert\mathcal{G}_n\right\vert}
\leq \varphi_\varepsilon(3D_0)<+\infty.
\end{equation*}
Since this upper bound does not depend on $R$, the condition of
\bref{eq:CondsoficIntegrability} is satisfied in this case.
\medskip

\noindent\textbf{Second case} We now assume that $s=(\mbfe,1)$. Using the above
inequalities and \cref{Lmm:ActionOnTheShifts} then \cref{Lmm:Enumeration}, we get
\begin{align*}
  \sum_{r=0}^R \varphi_\varepsilon
  &\left(r\right)\lim_{\mathcal{U}}
  \frac{\left\vert\left\{x\in \mathcal{G}_n^{(1)}\mid
  d_H \left(\iota_n(x),\iota_n\left(x\cdot s\right)\right)=r\right\}\right\vert}%
    {\left\vert\mathcal{G}_n\right\vert}\\
  &\leq\sum_{i\geq 0} \varphi_\varepsilon (3D_i) \lim_{\mathcal{U}}
    \frac{\left\vert \{x\in \calG^{(1)}_n \mid i_0(x)=i \} \right\vert}%
    {\left\vert \calG_n \right\vert}\\
  &\leq\sum_{i=0}^n \varphi_\varepsilon\left(3D_i\right)
    \frac{\kappa-1}{\kappa^{i}}.
  \end{align*}
  But,
  by \bref{Claim:Dn} there
  exists a constant $\CSnn \geq 1$ such that $D_i\leq \CSnn
  \kappa^{i}l_{\Landing(i)}$, for all $i\in \bN$.
  Therefore, using that $\rho$ is non-decreasing, then \bref{Rq:IneqTildeRho},
  we obtain that for all $i\in \bN$
  \begin{equation*}
    \rho(3D_i) \leq \rho\left(3\CSnn \kappa^{i}l_{\Landing(i)} \right)
    \leq 3\CSnn\rho\left( \kappa^{i}l_{\Landing(i)} \right).
  \end{equation*}
  Now, recall that the map $\rhoaff$ given by \bref{Prop:B2}
  verifies $\rhoaff\left( \kappa^{i}l_{\Landing(i)}
  \right)=\kappa^i$ and that there exists some constant $c\geq 1$ depending
  only on $\BZ$ such that $\rho\left( \kappa^{i}l_{\Landing(i)} \right)\leq c\rhoaff\left( \kappa^{i}l_{\Landing(i)}
  \right)$.
  Combining these inequalities and recalling that
  $\varphi_{\varepsilon}=\rho^{1-\varepsilon}$, we thus get
  \begin{align*}
    \sum_{i=0}^n \varphi_{\varepsilon}\left(3D_i\right)
    \frac{\kappa-1}{\kappa^{i}}
    &\leq %
      \sum_{i=0}^n  \left(3\CSnn c \kappa^i \right)^{(1-\varepsilon)}\frac{\kappa-1}{\kappa^{i}},\\
    &=
       (3\CSnn c)^{(1-\varepsilon)} (\kappa-1)\sum_{i=0}^n  \kappa^{-\varepsilon i}.
  \end{align*}
  Finally, since $\kappa\geq 3$, we have $\kappa^{-\varepsilon}<1$ and thus the
  sequence $(\kappa^{-\varepsilon i})_i$ is 
  summable. Hence the sum $\sum_{i=0}^n  \kappa^{-\varepsilon i}$ is bounded by
  above by a constant 
  that does not depend on $n$ nor on $R$. Therefore
  \cref{eq:CondsoficIntegrability} is verified.
  \medskip

  \noindent\textbf{Conclusion on the integrability.} By \bref{Th:soficME}, there
  exists an at most one-to-one $\varphi$-integrable measure subgroup coupling
  from $G$ to $H$.
\end{proof}

\bibliographystyle{alpha}
\bibliography{Biblio}
\vfill
\begin{flushleft}
  \href{https://aescalier.perso.math.cnrs.fr/index_eng.html}{Amandine Escalier}\\ 
  Lyon 1 Université, EC Lyon, INSA Lyon, Université Jean Monnet, CNRS, ICJ, UMR 5208, Villeurbanne, France\\
  \emph{e-mail:~}\texttt{amandine.escalier@math.cnrs.fr}
\end{flushleft}
\end{document}